\documentclass[11pt]{article}
\usepackage{paper, asb}

\usepackage{amsmath}
\usepackage{amsfonts}
\usepackage{amssymb}
\usepackage{amsthm}
\usepackage{hyperref}
\usepackage{comment}

\newtheorem{theorem}{Theorem}
\newtheorem{lemma}[theorem]{Lemma}
\newtheorem{corollary}[theorem]{Corollary}

\usepackage{array}
\usepackage{arydshln}
\newcommand{\papertitle}{A Flexible Gradient Tracking Algorithmic Framework for Decentralized Optimization}
\newcommand{\paperauthora}{Albert S. Berahas}
\newcommand{\paperauthoraaffiliation}{University of Michigan}
\newcommand{\paperauthoraemail}{albertberahas@gmail.com}
\newcommand{\paperauthorb}{Raghu Bollapragada}
\newcommand{\paperauthorbaffiliation}{University of Texas at Austin}
\newcommand{\paperauthorbemail}{raghu.bollapragada@utexas.edu}
\newcommand{\paperauthorc}{Shagun Gupta}

\newcommand{\paperauthorcemail}{shagungupta@utexas.edu}

\begin{document}
% generates the title
\title{\papertitle}

\author{\paperauthora\footnotemark[1]\ \footnotemark[3]
   \and \paperauthorb\footnotemark[2]
   \and \paperauthorc\footnotemark[2]}

\maketitle

\renewcommand{\thefootnote}{\fnsymbol{footnote}}
\footnotetext[1]{\paperauthoraaffiliation. (\url{\paperauthoraemail})}
\footnotetext[2]{\paperauthorbaffiliation. (\url{\paperauthorbemail,\paperauthorcemail})}
% \footnotetext[3]{\paperauthorbaffiliation. (\url{\paperauthorcemail})}
\footnotetext[3]{Corresponding author.}
\renewcommand{\thefootnote}{\arabic{footnote}}

%%%%%%%%%%%%%%%%%%%%%%%%%%%%
% Abstract
%%%%%%%%%%%%%%%%%%%%%%%%%%%%
\begin{abstract}{
In decentralized optimization over networks, each node in the network has a portion of the global objective function and the aim is to collectively optimize this function. Gradient tracking methods have emerged as a popular alternative for solving such problems due to their strong theoretical guarantees and robust empirical performance. These methods perform two operations (steps) at each iteration: (1) compute local gradients at each node, and (2) communicate local information with neighboring nodes. The complexity of these two steps can vary significantly across applications. In this work, we present a flexible gradient tracking algorithmic framework designed to balance the composition of communication and computation steps over the optimization process using a randomized scheme. The proposed framework is general, unifies gradient tracking methods, and recovers classical gradient tracking methods as special cases. We establish convergence guarantees in expectation and illustrate how the complexity of communication and computation steps can be balanced using the provided flexibility. Finally, we illustrate the performance of the proposed methods on quadratic and logistic regression problems, and compare against popular algorithms from the literature.

}
\end{abstract}

%%%%%%%%%%%%%%%%%%%%%%%%%%%%
% Body of Paper
%%%%%%%%%%%%%%%%%%%%%%%%%%%%
%%%%%%%%%%%%%%%%%%%%%%%%%%%%
% Introduction
%%%%%%%%%%%%%%%%%%%%%%%%%%%%
\section{Introduction} \label{sec.intro}

We consider decentralized optimization problems over networks. The goal is to minimize a global objective function over a network of agents (nodes), where each agent only has access to local information (local objective function) and each agent can only 
% communicate (exchange information) 
exchange information (communicate) with a subset of the other agents (neighbors). 
The nodes and edges (between neighboring nodes) form an undirected connected network, and information is exchanged along the edges. 
Due to data storage, data privacy and bandwidth consideration and limitations, such problems arise in many application areas, e.g., machine learning \cite{forero2010consensus,tsianos2012consensus}, multi-agent coordination \cite{cao2012overview, zhou2011multirobot}, sensor networks \cite{baingana2014proximal, predd2007distributed},  and signal processing \cite{combettes2011proximal, bazerque2009distributed}.

Algorithms for decentralized optimization typically %generally 
perform two types of steps: $(1)$ a computation step (e.g., the computation of the gradient of the local objective function at each node), 
and $(2)$ a communication step (e.g.,
%\sg{exchanging}
 the exchange of 
information among neighboring nodes). 
The complexity of these two steps can differ significantly across applications \cite{bertsekas2015parallel, mcmahan2017communication,nedic2018network}. 
For example, solving a large scale machine learning problem on a cluster of computers with shared memory access has a higher cost of computation than communication \cite{tsianos2012consensus}. 
% On the other hand, optimal channel allocation \cite{magnusson2017bandwidth} or multi agent coordination \cite{cao2012overview, zhou2011multirobot} over wireless sensor networks require economic usage of communications due to the limited power of sensors. 
Conversely, economical usage of communication resources is required for tasks such as optimal channel allocation \cite{magnusson2017bandwidth} or multi-agent coordination \cite{cao2012overview, zhou2011multirobot} in wireless sensor networks, due to limited power resources.
Thus, overall efficiency 
is contingent upon striking a balance between the number of communication and computation steps. In this paper, we present a randomized gradient tracking algorithmic framework (\texttt{RGTA}) that is endowed with flexibility in the composition of computation and communication steps in order to achieve this balance.

The \texttt{RGTA} framework decomposes the communication and computation steps similar to \cite{berahas2018balancing,sayed2014diffusion,chen2012fast}, and introduces flexibility via a randomized scheme that dictates the composition of these two steps. This results in a less rigid framework in comparison to deterministic approaches that restrict each iteration to the user specified number 
of communication and computation steps \cite{berahas2018balancing, berahas2023balancing, mansoori2021flexpd,chen2012fast,sayed2014diffusion}. 
\texttt{RGTA} also builds upon the unifying framework for gradient tracking methods introduced in \cite{berahas2023balancing}, a class of decentralized optimization methods shown to have sound theoretical guarantees and robust and efficient empirical performance in various decentralized settings \cite{nedic2017achieving,shi2015extra,xu2015augmented,nedic2017geometrically,sun2022distributed,berahas2023balancing, pu2020push,shah2023adaptive}. 
This unified framework in conjunction with the proposed randomized scheme allows for a 
direct theoretical and empirical comparison of popular gradient tracking methods.

\subsection{Literature Review} \label{sec.lit}

Gradient tracking algorithms, e.g., DIGing \cite{nedic2017achieving}, EXTRA \cite{shi2015extra}, AugDGM \cite{xu2015augmented}, ATC-DIGing \cite{nedic2017geometrically}, and SONATA \cite{sun2022distributed}, form one of the most popular class 
% classes 
of algorithms for decentralized optimization due to their strong % and sound 
theoretical guarantees and superior empirical performance compared to alternative % classes of 
first-order algorithms. 
These methods have been %modified and 
extended to important decentralized settings, e.g., stochastic~\cite{pu2021distributed}, time varying networks \cite{nedic2017achieving}, directed networks \cite{nedic2017achieving, pu2020push}.
% and others.
Moreover, variants that employ gossip-type communication strategies have also been developed~\cite{pu2020push, pu2021distributed}, where all pairs of connected nodes randomly and independently exchange information pairwise.
We note that this is different from the randomized scheme considered in this work. We consider a randomized scheme similar to that in federated learning settings \cite{mcmahan2017communication, karimireddy2020scaffold, mitra2021linear, mishchenko2022proxskip}, i.e., the decision to communicate is based on a randomized scheme and all nodes synchronously exchange information.

Motivated % and inspired 
by applications, the subject of designing decentralized optimization algorithms that explicitly provide flexibility with respect to the number of communication and computation steps has gained significant attention in recent years. For example, in \cite{berahas2018balancing, 9479747, berahas2023balancing, mansoori2021flexpd,chen2012fast,berahas2019nested}, deterministic schemes were 
proposed that allow for the exact specification of the composition of the two steps. 
On the other hand, randomized schemes have also been proposed; see e.g., FedAvg \cite{mcmahan2017communication}, FedLin \cite{mitra2021linear}, Scaffold \cite{karimireddy2020scaffold}, Scaffnew \cite{mishchenko2022proxskip}, S-local-GD \cite{gorbunov2021local}. 
These methods are intended for settings in which the communication steps are expensive and aim to reduce the total rounds of communication via a randomized scheme.  
That said, they are primarily designed for fully connected networks (a network where all pairs of nodes are connected) using a client-server communication setup \cite{mcmahan2017communication,bertsekas2015parallel}. In this work, we present a randomized algorithmic framework for the general decentralized setting (the aforementioned setting is a special case), i.e., general connected undirected networks.

\subsection{Contributions} \label{sec.contri}

Our main contributions can be summarized as follows:

\begin{enumerate}
    \item We propose \texttt{RGTA}, a randomized variant of the gradient tracking algorithmic framework introduced in~\cite{berahas2023balancing} that \emph{probabilistically} balances the number of communication and computation steps in decentralized optimization. As a special case, the \texttt{RGTA} framework recovers the multiple communications and single computation setting of the deterministic framework proposed in~\cite{berahas2023balancing}. Moreover, we discuss three instances of \texttt{RGTA} that recover  
    popular gradient tracking methods, e.g., \texttt{RGTA-1} \cite{shi2015extra, nedic2017achieving}, \texttt{RGTA-2} \cite{di2016next, sun2022distributed} and \texttt{RGTA-3} \cite{nedic2017geometrically, xu2015augmented}; see \cref{tab: Algorithm_Def}.
    \item We establish a global linear rate of convergence in expectation for the \texttt{RGTA} framework under sufficient conditions on step size parameter and information exchanged in the network, and illustrate the advantages of \texttt{RGTA-3} compared to \texttt{RGTA-2} (and \texttt{RGTA-2} compared to \texttt{RGTA-1}). 
    We provide worst-case complexity results for the number of computation steps and the expected number of communication steps required to achieve an $\epsilon-$accurate solution, and quantify the trade-off between communication and computation. We show a reduction in number of computation steps by performing multiple communication steps, and reduction in expected number of communication steps by performing communication steps less often in a randomized manner. Note that due to our randomized scheme, our communication complexity results are in expectation. %we present the expected communication complexity instead of communication complexity due to our randomized manner of performing communications.}
    % We also provide worst-case complexity results for
    % the number of computation steps and the expected number of communication steps required to achieved an $\epsilon-$accurate solution, and quantify the tradeoff between communication and computation. \asb{We note that we present the expected communication complexity instead of communication complexity due to randomly selecting the communications steps at every iteration.} \sg{Specifically, we show reduction in number of computation steps with the increase in communication until we reach gradient descent performance and reduction in expected number of communication steps via performing communications less often in a randomized manner.}
    % Specifically, we show reduction in gradient complexity with the increase in communication steps until we reach gradient descent complexity and reduction in communication complexity in expectation via performing communications less often in a randomized manner.
    \item We illustrate the empirical performance of \texttt{RGTA} on quadratic and binary classification logistic regression problems. We % illustrate 
    demonstrate the effects and benefits of employing the randomized scheme to reduce the 
    number of communication or computation steps, as desired, in \texttt{RGTA-1}, \texttt{RGTA-2} and \texttt{RGTA-3}. We also demonstrate that our proposed randomized framework is competitive with several popular and efficient methods, e.g.,  FedAvg \cite{mcmahan2017communication}, Scaffold \cite{karimireddy2020scaffold} and Scaffnew \cite{mishchenko2022proxskip}.
\end{enumerate} 

\subsection{Paper Organization} 
The paper is organized as follows. In~\cref{sec.methods}, we formally define the problem and notation (\cref{sec.formulation}), and present our proposed randomized gradient tracking algorithmic framework (\texttt{RGTA}, \cref{sec.algorithm}). In \cref{sec.theory}, we establish a linear rate of convergence in expectation, perform a theoretical comparison among popular gradient tracking methods, and provide worst case complexity results for the number of computation and the expected number of communication steps required to achieve $\epsilon$-accurate solutions. %in the \texttt{RGTA} framework (\cref{sec.compl}). 
We present numerical experiments on quadratic and binary classification logistic regression problems, and compare \texttt{RGTA} with popular algorithms from the literature in \cref{sec.num_exp}. Finally, we provide concluding remarks in \cref{sec.conc}.

%%%%%%%%%%%%%%%%%%%%%%%%%%%%
% Methods
%%%%%%%%%%%%%%%%%%%%%%%%%%%%
\section{Problem Formulation and Algorithm}\label{sec.methods}
In this section, we formulate the decentralized optimization problem and define the notation used throughout the paper (\cref{sec.formulation}). Then, we describe our proposed algorithmic framework (\cref{sec.algorithm}). 

\subsection{Decentralized Optimization Problem} \label{sec.formulation}

The decentralized optimization problem considered in this paper is defined as
\begin{align}		\label{eq: prob}
	\min_{x\in \mathbb{R}^d}\quad f(x) = \frac{1}{n} \sum_{i=1}^n f_i(x),
\end{align}
where $x\in \mathbb{R}^d$ is the decision variable, $f_i: \mathbb{R}^d \rightarrow \mathbb{R}$ is the local objective function only known to node $i$ (for all $i\in \{1,2,...,n \}$), and $f: \mathbb{R}^d \rightarrow \mathbb{R}$ is the global objective function. 
The goal is to minimize the global objective function over a network of nodes, where each node only has access to local information (local objective function) and can only exchange information with neighboring nodes in the network. 

Problem~\eqref{eq: prob} is often reformulated as a consensus optimization problem by introducing a local copy of the decision variable at each node, $x_i \in \mathbb{R}^d$ (for all $i\in \{1,2,...,n \}$), and adding a consensus constraint that ensures that the decision variable at each node is the same as its neighbors, i.e., $x_i = x_j$ for all $(i, j) \in \mathcal{E}$ where $\mathcal{E}$ denotes the set of edges in the network; see e.g.,~\cite{bertsekas2015parallel,nedic2017achieving}. This reformulation can be compactly represented as,
\begin{equation}\label{eq: cons_prob_mixing}
\begin{aligned}		
	\min_{x_i \in \mathbb{R}^d}&\quad \textbf{f} (\textbf{x}) = \frac{1}{n} \sum_{i=1}^n f_i(x_i)\\
	\text{s.t.} & \quad (\textbf{W}\otimes I_d)\textbf{x} = \textbf{x}, 
\end{aligned}
\end{equation}
where $\textbf{x} \in \mathbb{R}^{nd}$ is a concatenation of all local copies $x_i \in \mathbb{R}^{d}$ (for all $i\in \{1,2,...,n \}$), $\textbf{W} \in \mathbb{R}^{n \times n}$ is a mixing matrix that captures the connectivity of the network, $I_d \in \mathbb{R}^{d \times d}$ is the identity matrix, and $\otimes$ denotes the Kronecker product. The mixing matrix $\textbf{W}$ is a symmetric, doubly-stochastic matrix with $w_{ii}>0$ and $w_{ij}>0$ ($i\neq j$) if and only if $(i, j) \in \mathcal{E}$. 
% satisfies the following: $(1)$ symmetric, $(2)$ doubly-stochastic, and $(3)$ $w_{ii}>0$ and $w_{ij}>0$ ($i\neq j$) if and only if $(i, j) \in \mathcal{E}$. 
Given the network is connected, problems \eqref{eq: prob} and \eqref{eq: cons_prob_mixing} are equivalent.

Moving forward, the local copies of the decision and auxiliary variables for node $i$ at iteration $k$ are denoted by $x_{i, k} \in \mathbb{R}^d$ and $y_{i, k} \in \mathbb{R}^d$, respectively. The (component-wise) average of all local decision and  auxiliary variables are denoted by $\bar{x}_{k} = \frac{1}{n} \sum_{i=1}^n x_{i, k}$ and $\bar{y}_{k} = \frac{1}{n} \sum_{i=1}^n y_{i, k}$, respectively. Boldface lowercase letters represent the concatenation of vectors, i.e.,   
\begin{align*}
    \xmbf_{k} = 
    \begin{bmatrix}
        x_{1, k}\\
        x_{2, k}\\
        \vdots \\
        x_{n, k}
    \end{bmatrix} \in \mathbb{R}^{nd}\mbox{,} \quad
    \ymbf_{k} = 
    \begin{bmatrix}
        y_{1, k}\\
        y_{2, k}\\
        \vdots \\
        y_{n, k}
    \end{bmatrix} \in \mathbb{R}^{nd}
    \mbox{,} \quad
      \nabla \fmbf(\xmbf_{k}) = 
    \begin{bmatrix}
        \nabla f_1(x_{1, k})\\
        \nabla f_2(x_{2, k})\\
        \vdots \\
        \nabla f_n(x_{n, k})
    \end{bmatrix} \in \mathbb{R}^{nd}.%,
\end{align*}
The vectors $\xbb_{k}$ and $\ybb_{k}$ represent concatenations of $n$ copies of the average local decision ($\Bar{x}_{k}$) and auxiliary ($\Bar{y}_{k}$) variables, respectively. The $d$ dimensional vector of all ones is denoted by $1_d$. The spectral radius of a square matrix $A$ is $\rho(A)$. Lastly, matrix inequalities are defined component-wise.

%%%%%%%%%%%%%%%%%%%%%%%%%%%%%%%%%%%%%%%%%%%%%%%%%%%%%%%%%%%%%%%%%%%%%%%%%%%%%%%%%%%
%%%%%%%%%%%%%%%%%%%%%%%%%%%%%%%%%%%%%%%%%%%%%%%%%%%%%%%%%%%%%%%%%%%%%%%%%%%%%%%%%%%
%%%%%%%%%%%%%%%%%%%%%%%%%%%%%%%%%%%%%%%%%%%%%%%%%%%%%%%%%%%%%%%%%%%%%%%%%%%%%%%%%%%

\subsection{A Randomized Gradient Tracking Framework} \label{sec.algorithm}

In this section, we describe our proposed randomized gradient tracking algorithmic framework (\texttt{RGTA}) that provides flexibility in the number of communication and computation steps. 
The framework is built around the following generalized iterate update form for gradient tracking methods %(see \cite{berahas2023balancing})\sg{,}
\begin{equation}\label{eq: general_form}
\begin{aligned}
    \xmbf_{k+1} & = \Zmbf_1 \xmbf_{k} - \alpha \Zmbf_2 \ymbf_{k} \\ 
    \ymbf_{k+1} & = \Zmbf_3 \ymbf_{k} + \Zmbf_4 (\nabla \fmbf(\xmbf_{k+1}) - \nabla \fmbf(\xmbf_{k})), 
\end{aligned}
\end{equation}
where $\alpha > 0$ is the step size, $\Zmbf_i = \Wmbf_i \otimes I_d \in R^{nd \times nd}$ and $\Wmbf_i \in R^{n \times n}$ are communication matrices, for all $i = 1, 2, 3, 4$; see \cite{berahas2023balancing}.
% (for all $i = 1, 2, 3, 4$) and $\Wmbf_i \in R^{n \times n}$ are communication matrices. 
A communication matrix represents a network topology using all nodes and a subset of the edges in the network over which information can be exchanged. Formally, a communication matrix $\Vmbf$ is a symmetric, doubly stochastic matrix such that $v_{ii} > 0$, $v_{ij} \geq 0$ if $(i, j) \in \mathcal{E}$ and $v_{ij} = 0$ if $(i, j) \notin \mathcal{E}$. The iterate update form given in \eqref{eq: general_form} encompasses several popular gradient tracking methods. Some examples are summarized in \cref{tab: Algorithm_Def}.
\begin{table}[ht]
\centering
\caption{Special Cases of Randomized Gradient Tracking Algorithmic Framework (\texttt{RGTA}).} \label{tab: Algorithm_Def}
\begin{tabular}{l*{4}{>{\centering\arraybackslash}p{0.5cm}}cc}\toprule
\multirow{2}{*}{Method} &\multicolumn{4}{c}{Communication Matrices} & Algorithms in literature & \texttt{GTA} \cite{berahas2023balancing}\\\cmidrule{2-5}
&$\Wmbf_1$&$\Wmbf_2$&$\Wmbf_3$&$\Wmbf_4$& ($n_c = 1$, $p = 1$) & ($n_g = 1$, $p=1$ )\\\midrule
\texttt{RGTA-1} &$\Wmbf$ &$I_n$ &$\Wmbf$ &$I_n$& DIGing \cite{nedic2017achieving}, EXTRA %\footnotemark[1] 
\cite{shi2015extra} & \texttt{GTA-1}  \\\hdashline
\texttt{RGTA-2} &$\Wmbf$ &$\Wmbf$ &$\Wmbf$ &$I_n$ & SONATA \cite{sun2022distributed}, NEXT \cite{di2016next,pu2020push} & \texttt{GTA-2}\\\hdashline
\texttt{RGTA-3} &$\Wmbf$ &$\Wmbf$ &$\Wmbf$ &$\Wmbf$ & Aug-DGM \cite{xu2015augmented}, ATC-DIGing \cite{nedic2017geometrically} & \texttt{GTA-3}\\ 
\bottomrule
\end{tabular}
Note: $\Wmbf$ is a mixing matrix.
\end{table}

The iterate updates given in \eqref{eq: general_form} perform a single communication step and a single computation step per iteration. The \texttt{RGTA} framework is more versatile and allows for flexibility in the number of communication steps per iteration. The framework is presented in \cref{alg: Randomised}. The key idea is to allow for multiple communication steps per iteration, but to communicate less often, based on a randomized scheme. In each iteration, independently, the \texttt{RGTA} approach either performs $n_c \geq 1$ communication steps with probability $p \in (0, 1]$ or uses only local information to update the local decision and auxiliary variables with probability $1-p$. The multiple communication steps are represented by replacing $\Zmbf_i$ with $\Zmbf_i ^ {n_c} = \Wmbf_i^{n_c} \otimes I_d$ for all $i=1, 2, 3, 4$, which translates to performing $n_c$ consensus steps. Therefore, changing $p$ and $n_c$ allows one to achieve any composition of communication and computation steps in expectation.

\begin{algorithm}[H]
    \caption{\texttt{RGTA} - Randomised Gradient Tracking Algorithm}
    \textbf{Inputs :} initial point $\xmbf_0 \in \R{nd}$, step size $\alpha > 0$, communication steps $n_c \geq 1$, \\ communication probability $p \in (0, 1]$.
    \begin{algorithmic}[1]
        \State $\textbf{y}_0 \gets \nabla \textbf{f}(\textbf{x}_0)$
        \For{$k \gets 0, 1, 2$ ... }
            \State Choose a sample $\theta_k \sim \mbox{Bernouli}(p)$ ($\mathbb{P}\left[\theta_k = 1\right] = p$).%Flip a coin $\theta_k$ where $P(\theta_k = 1) = p$
            \If{$\theta_k = 1$}
                \State $\textbf{x}_{k+1} \gets \textbf{Z}_1^{n_c} \textbf{x}_k - \alpha \, \textbf{Z}_2^{n_c} \textbf{y}_k$
                \State $\textbf{y}_{k+1} \gets \textbf{Z}_3^{n_c} \textbf{y}_k + \textbf{Z}_4^{n_c}(\nabla \textbf{f}(\textbf{x}_{k+1})  - \nabla \textbf{f}(\textbf{x}_k))$
            \Else
            \State $\textbf{x}_{k+1} \gets \textbf{x}_k - \alpha \textbf{y}_k$
            \State $\textbf{y}_{k+1} \gets \textbf{y}_k + \nabla \textbf{f}(\textbf{x}_{k+1})  - \nabla \textbf{f}(\textbf{x}_k)$
            \EndIf
        \EndFor
    \end{algorithmic}
    \label{alg: Randomised}
\end{algorithm}

\bremark 
We make the following remarks about \cref{alg: Randomised}. 
\begin{itemize}
    \item \textbf{Communications and Computations:} In each iteration, the algorithm performs one computation step (computation of local gradients, $\nabla \fmbf (\xmbf_{k+1})$; Lines 6 and 9), and $n_c$ communication steps with probability $p \in (0, 1]$, independently. Thus, in expectation, the algorithm performs $p n_c$ communication steps every iteration and $\tfrac{1}{p}$ local updates (computation steps) between iterations performing communication steps.
    \item \textbf{Step size ($\alpha>0$):} The algorithm uses a constant step size that depends on problem parameters, network parameters, the choice of $n_c$ and $p$, and the communication matrices, i.e., $\Wmbf_i$ for all $i=1, 2, 3, 4$.
    \item \textbf{Flexibility:} The algorithm has two user defined parameters ($p$ and $n_c$) that dictate the composition of communication and computation steps in expectation. 
    As a result, 
    the number of local updates
    % the gap between communication steps 
    is not constant and follows a geometric distribution with probability $p$. 
    This is different from the deterministic approaches in the literature \cite{berahas2018balancing, 9479747, berahas2023balancing, mansoori2021flexpd,chen2012fast,berahas2019nested}, that constrain each iteration to a fixed user-specified composition. Thus, our approach is more flexible than the aforementioned approaches. %defining a fixed composition for every iteration.}
    % Our approach is less rigid and more flexible, and several of the deterministic compositions are indeed special cases of our framework. 
    %\sg{The algorithm achieves the user specified composition of communication and computation steps (based on the choice of $p$ and $n_c$) in expectation. Nevertheless, the gap between iterations that perform communications is not constant. They follow a geometric distribution with probability $p$. This produces a framework less rigid than deterministic approaches from literature  \cite{berahas2018balancing, 9479747, berahas2023balancing, mansoori2021flexpd,chen2012fast,berahas2019nested} that constraint each iteration to the composition specified by the user.}
\end{itemize} 
\eremark

The choice of the gradient tracking method used within the \texttt{RGTA} framework significantly impacts the performance. The \texttt{RGTA} framework is built around the general iterate update form \eqref{eq: general_form}, and as a result, we are able to establish a unified theory for gradient tracking methods when employed with the proposed randomized scheme. This allows us to perform a direct theoretical and empirical comparison of different gradient tracking methods when employed within the \texttt{RGTA} framework. We demonstrate this for the special cases presented in Table~\ref{tab: Algorithm_Def} (\texttt{RGTA-1}, \texttt{RGTA-2} and \texttt{RGTA-3}).

%%%%%%%%%%%%%%%%%%%%%%%%%%%%
% Theoretical Results
%%%%%%%%%%%%%%%%%%%%%%%%%%%%
\section{Convergence Analysis}\label{sec.theory}

In this section, we establish theoretical convergence guarantees in expectation
for our proposed algorithmic framework
% algorithm 
\texttt{RGTA}. We present conditions for achieving a linear rate of convergence, and explicitly quantify the effect of both the number of communication steps ($n_c$) and the communication probability ($p$) on the rate of convergence. 
Moreover, we %then
analyze the three methods described in \cref{tab: Algorithm_Def},  \texttt{RGTA-1}, \texttt{RGTA-2} and \texttt{RGTA-3}, as special cases and provide a theoretical comparison. We assume $p < 1$ in our general analysis, and state and discuss the results for the special case of $p = 1$ separately. We then provide a complexity analysis that demonstrates the effect of employing $n_c$ and $p$ parameters to achieve reduction in communication or computation complexity.
%, in expectation.}
We make the following assumptions about %on
the objective function.
\bassumption    \label{asum: convex_and_smooth}
    The global objective function $f : \mathbb{R}^d \rightarrow \mathbb{R}$ is $\mu$-strongly convex. Each component function $f_i : \mathbb{R}^d \rightarrow \mathbb{R}$ (for all $i \in  \{1, 2, ..., n\}$) has L-Lipschitz continuous gradients. That is, for all $z, z' \in \mathbb{R}^d$, 
    \begin{align*}
        f(z')  &\geq f(z) +  \langle \nabla f(z), z' - z \rangle + \tfrac{\mu}{2} \|z' - z\|_2^2,    \\
        \|\nabla f_i(z) - \nabla f_i(z')\|_2 &\leq L\|z - z'\|_2,    \qquad \qquad \forall \,\, i = 1, ..., n.
    \end{align*}
\eassumption

Under \cref{asum: convex_and_smooth}, \eqref{eq: prob} has unique minimizer denoted by $x^* \in \mathbb{R}^d$. We define, 
\begin{align*}
   \left\|\Wmbf^{n_c} - \frac{1_n1_n^T}{n}\right\|_2 = \beta^{n_c}, \quad \mbox{and} \quad \left\|\Wmbf_i^{n_c} - \frac{1_n1_n^T}{n}\right\|_2 = \beta_i^{n_c}, \quad \forall \,\, i = 1, 2, 3, 4,    
\end{align*}
where $\beta \in [0, 1)$ as $\Wmbf$ is a mixing matrix for a connected network and $\beta_i \in [0,1]$ (for all $i = 1, 2, 3, 4$) as $\Wmbf_i$'s are symmetric, doubly-stochastic communication matrices. Using the above definitions, and $\Zmbf^{n_c}= \Wmbf^{n_c}\otimes I_{d}$ and $\Zmbf_i^{n_c} = \Wmbf_i^{n_c}\otimes I_{d}$ for all $i = 1, 2, 3, 4$, it follows that %we have
\begin{align}   \label{eq: beta_and_Z}
    \|\Zmbf^{n_c} - \Imbf\|_2 = \beta^{n_c}, \quad \mbox{and} \quad  \|\Zmbf_i^{n_c} - \Imbf\|_2 = \beta_i^{n_c}, \,\, \forall \,\, i = 1, 2, 3, 4,
\end{align}
where $\Imbf = \frac{1_n1_n^T}{n} \otimes I_d$. We also define,
\begin{align}   \label{eq: derivative_terms_define}
   g_{k} = \frac{1}{n} \sum_{i = 1}^n \nabla f_i(x_{i, k}), \quad \mbox{and} 
 \quad \gbar_{k} = \frac{1}{n} \sum_{i = 1}^n \nabla f_i(\xbar_{k}), %\quad \mbox{and} \quad \Imbf = \frac{1_n1_n^T}{n} \otimes I_d,
\end{align}
where $x_{i,k}$ is the local copy of node %agent 
$i$ at  iteration $k$ and $\bar{x}_{k} = \frac{1}{n} \sum_{i=1}^n x_{i, k}$. 

The iterate update form of \cref{alg: Randomised} can be expressed as 
\begin{equation}\label{eq: iter_form_simple}
\begin{aligned} 
    \xmbf_{k+1} &= \Rmbf_1^{n_c} \xmbf_{k} - \alpha  \Rmbf_2^{n_c} \ymbf_{k}, \\
    \ymbf_{k+1} &= \Rmbf_3^{n_c} \ymbf_{k} + \Rmbf_4^{n_c} \left(\nabla \fmbf(\xmbf_{k+1}) - \nabla \fmbf(\xmbf_{k})\right),
\end{aligned}
\end{equation}
where $\Rmbf_i^{n_c} = \Smbf_i^{n_c} \otimes I_d \in \mathbb{R}^{nd \times nd}$ and $\Smbf_i \in \mathbb{R}^{n\times n}$ is a random matrix for all $i = 1, 2, 3, 4$. The random matrices $\Smbf_i$ are defined as
\begin{align*}
   (\Smbf_1, \Smbf_2, \Smbf_3, \Smbf_4)  = 
  \begin{cases}
    (\Wmbf_1, \Wmbf_2, \Wmbf_3, \Wmbf_4) & \text{with probability} \quad p, \\
    (I_n, I_n, I_n, I_n) & \text{with probability} \quad 1 - p. 
  \end{cases}
\end{align*}
Thus, these matrices are always symmetric and doubly stochastic. By \eqref{eq: beta_and_Z} and $\left\|I_{nd} - \Imbf\right\|_2 = 1$, it follows for all $i=1, 2, 3, 4$, 
\begin{align}   \label{eq: R_expected_norm}
    \E \left[\left\|\Rmbf_i^{n_c} - \Imbf\right\|_2\right] = p \left\|\Zmbf_i^{n_c} - \Imbf\right\|_2 + (1-p)\left\|I_{nd} - \Imbf\right\|_2 = p \beta_i^{n_c} + (1-p).
\end{align}

For the analysis, we define the expected error vector
\begin{align} \label{eq: error_vector_def}
    r_k = \E \begin{bmatrix}
        \|\xbar_{k} - x^*\|_2\\
        \|\xmbf_{k} - \Bar{\xmbf}_{k}\|_2\\
        \|\ymbf_{k} - \Bar{\ymbf}_{k}\|_2\\
    \end{bmatrix}, \quad \forall \,\, k \geq 0,
\end{align}
which combines the expected optimization error $\E[\|\xbar_{k} - x^*\|_2]$ and the expected consensus errors, $\E[\|\xmbf_k - \xbb_k\|_2]$ and $\E[\|\ymbf_k - \ybb_k\|_2]$, at iteration $k$. We show convergence in expectation by establishing conditions for a monotonic decrease in the norm of the expected error vector \eqref{eq: error_vector_def}. %$\|r_k\|_2$. 
We first establish a technical lemma that quantifies the progression of $r_k$. 

\begin{lemma}\label{lem: lyapunov}
    Suppose \cref{asum: convex_and_smooth} holds and $\alpha \leq \frac{1}{L}$% in \cref{alg: Randomised}
    . Then, for all $k\geq 0$,
    \begin{align}\label{eq: general_r}
        r_{k+1} \leq A(n_c, p) r_k,%  \quad \forall \,\,  k \geq 0, 
    \end{align}
    \begin{align} \label{eq: general_A}
    \mbox{where}\qquad A(n_c, p) = \begin{bmatrix}
        1 - \alpha \mu & \frac{\alpha L}{\sqrt{n}} & 0\\
        0 & \eta_1 & \alpha\eta_2\\
        \sqrt{n} \eta_4 \alpha  L^2 & p\beta_4^{n_c}L\|\Zmbf_1^{n_c}-I_{nd}\|_2 + \eta_4 \alpha L^2  & \eta_3 + \eta_4 \alpha L\\
        \end{bmatrix}
    \end{align}
    and $\eta_i = p\beta_i^{n_c} + (1 - p)$ for all $i = 1, 2, 3, 4$.
\end{lemma}
\bproof
By \eqref{eq: iter_form_simple}, since $\Smbf_i$ is a doubly stochastic matrix for all $i=1, 2, 3, 4$, the average iterates can be expressed as,
\begin{align*}
    \Bar{x}_{k+1} & = \Bar{x}_k - \alpha \Bar{y}_k,   \\
    \Bar{y}_{k+1} & = \Bar{y}_k + g_{k+1} - g_{k}, 
\end{align*}
where $g_k$ is defined in \eqref{eq: derivative_terms_define}.
Taking the telescopic sum of $\Bar{y}_{i}$ from $i=0$ to $k$ with $\bar{y}_0 = g_0$ yields,
\begin{align}
    \Bar{y}_{k} & = g_k. \label{eq: y_bar_telescope}
\end{align}

We first consider the optimization error of the average iterates. That is, 
\begin{align*}
    \|\Bar{x}_{k+1} - x^*\|_2 & = \left\|\Bar{x}_k - \alpha \Bar{y}_k + \alpha \gbar_k - \alpha\gbar_k - x^*\right\|_2  \\
    & \leq \left\|\Bar{x}_k- \alpha\gbar_k - x^*\right\|_2 + \alpha \left\|\Bar{y}_k - \gbar_k\right\|_2  \\
    &\leq  (1-\alpha \mu) \|\Bar{x}_k - x^*\|_2 + \alpha \left\|g_k - \gbar_k\right\|_2 \\
    &= (1-\alpha \mu) \|\Bar{x}_k - x^*\|_2 + \frac{\alpha}{n}\left\|\sum_{i=1}^n \nabla f_i(x_{i, k}) - \nabla f_i(\Bar{x}_{k})\right\|_2 \\
    & \leq (1-\alpha \mu) \|\Bar{x}_k - x^*\|_2 + \frac{\alpha L}{n} \sum_{i=1}^n  \| x_{i, k} - \Bar{x}_{k}\|_2    \\
    & \leq (1-\alpha \mu) \|\Bar{x}_k - x^*\|_2 + \frac{\alpha L}{\sqrt{n}}  \| \xmbf_{k} - \Bar{\xmbf}_{k}\|_2,    \numberthis \label{eq : opt_bound_before_expectation}
\end{align*}
where the first inequality is due to the triangle inequality, the second inequality is obtained by performing one gradient descent iteration on the function $f$ under \cref{asum: convex_and_smooth} at the average iterate $\bar{x}_k$ with $\alpha \leq \frac{1}{L}$ \cite[Theorem 2.1.14]{nesterov1998introductory} and substituting \eqref{eq: y_bar_telescope}, the equality follows from \eqref{eq: derivative_terms_define}, the second to last inequality follows from \cref{asum: convex_and_smooth}, and the last inequality is due to $\sum_{i=1}^n \|x_{i, k} - \xbar_k\|_2 \leq \sqrt{n}\|\xmbf_{k} - \Bar{\xmbf}_{k}\|_2$. 
By \eqref{eq : opt_bound_before_expectation}, the optimization error can be bounded in expectation as,
\begin{align}
    \E\left[ \|\Bar{x}_{k+1} - x^*\|_2 \right] \leq (1-\alpha \mu) \E \left[ \|\Bar{x}_k - x^*\|_2 \right] + \frac{\alpha L}{\sqrt{n}}  \E \left[ \| \xmbf_{k} - \Bar{\xmbf}_{k}\|_2 \right].   \label{eq : opt_bound}
\end{align}

Next, we consider the consensus error in $\xmbf_k$,
\begin{align*}
    \xmbf_{k+1} - \Bar{\xmbf}_{k+1} &= \Rmbf_1^{n_c}\xmbf_k - \Bar{\xmbf}_{k} - \alpha \Rmbf_2^{n_c}\ymbf_k + \alpha \Bar{\ymbf}_k \\
    & = \Rmbf_1^{n_c}\xmbf_k - \Rmbf_1^{n_c}\Bar{\xmbf}_{k} - \alpha \Rmbf_2^{n_c}\ymbf_k + \alpha \Rmbf_2^{n_c}\Bar{\ymbf}_k  - \Imbf (\xmbf_k - \Bar{\xmbf}_k) + \Imbf (\ymbf_k - \Bar{\ymbf}_k)\\
    & = \left(\Rmbf_1^{n_c} - \Imbf\right)(\xmbf_k - \Bar{\xmbf}_k) - \alpha\left(\Rmbf_2^{n_c} - \Imbf\right)(\ymbf_k - \Bar{\ymbf}_k), 
\end{align*}
where the second equality follows from the fact that $\Smbf_1$ and $\Smbf_2$ are doubly stochastic matrices and adding $-\Imbf(\xmbf_k - \xbb_k) = 0$ and $-\Imbf(\ymbf_k - \ybb_k) = 0$. By the triangle inequality, 
\begin{align}\label{eq : g = 1 x con error} 
    \|\xmbf_{k+1} - \Bar{\xmbf}_{k+1}\|_2 & \leq \left\|\Rmbf_1^{n_c} - \Imbf\right\|_2 \|\xmbf_k - \Bar{\xmbf}_k\|_2 + \alpha \left\|\Rmbf_2^{n_c} - \Imbf\right\|_2 \|\ymbf_k - \Bar{\ymbf}_k\|_2.  
    % & \leq \beta_1^{n_c} \|\xmbf_k - 1\Bar{\xmbf}_k\|_2 + \alpha \beta_2^{n_c} \|\ymbf_k - \Bar{\ymbf}_k\|_2   \numberthis \label{eq : g = 1 x con error}  
\end{align}
The conditional expectation given $\xmbf_k$ and $\ymbf_k$ of the consensus error in \eqref{eq : g = 1 x con error} can be bounded as, 
\begin{align*}
    \E \left[ \|\xmbf_{k+1} - \Bar{\xmbf}_{k+1}\|_2 | \xmbf_k, \ymbf_k \right] & \leq \E \left[\left\|\Rmbf_1^{n_c} - \Imbf\right\|_2 | \xmbf_k, \ymbf_k \right] \|\xmbf_k - \Bar{\xmbf}_k\|_2\ \\
    & \quad + \alpha \E\left[\left\|\Rmbf_2^{n_c} - \Imbf\right\|_2| \xmbf_k, \ymbf_k \right] \|\ymbf_k - \Bar{\ymbf}_k\|_2   \\
    &= \eta_1 \|\xmbf_k - \Bar{\xmbf}_k\|_2 + \alpha \eta_2 \|\ymbf_k - \Bar{\ymbf}_k\|_2,
\end{align*}
where the equality follows from \eqref{eq: R_expected_norm} as $\Rmbf_1$ and $\Rmbf_2$ are independent of $\xmbf_k$ and $\ymbf_k$. 
% \sg{The above bound in expectation w.r.t. $\xmbf_k$ and $\ymbf_k$ yields,}
Taking the total expectation of the above (with respect to all the random variables starting with $\xmbf_0, \ymbf_0$) yields, 
\begin{align*}
    \E \left[ \|\xmbf_{k+1} - \Bar{\xmbf}_{k+1}\|_2\right] & \leq \eta_1 \E\left[\|\xmbf_k - \Bar{\xmbf}_k\|_2\right] + \alpha \eta_2 \E\left[\|\ymbf_k - \Bar{\ymbf}_k\|_2\right]. \numberthis \label{eq: con_x_bound}
\end{align*}

Finally, we consider the consensus error in $\ymbf_k$. That is,% By triangle inequality and adding $-\Imbf(\ymbf_k - \ybb_k) = 0$, we get 
\begin{align*}
    &\left\|\ymbf_{k+1} - \Bar{\ymbf}_{k+1}\right\|_2 \\
    =& \left\|\Rmbf_3^{n_c}\ymbf_k - \Bar{\ymbf}_{k} + \Rmbf_4^{n_c} (\nabla \fmbf(\xmbf_{k+1}) - \nabla \fmbf(\xmbf_{k})) - \Imbf (\nabla \fmbf(\xmbf_{k+1}) - \nabla \fmbf(\xmbf_{k}))\right\|_2 \\
    \leq & \left\|\Rmbf_3^{n_c} - \Imbf\right\|_2 \left\|\ymbf_k - \Bar{\ymbf}_k\right\|_2 + \left\|\Rmbf_4^{n_c} - \Imbf\right\|_2\left\|\nabla \fmbf(\xmbf_{k+1}) - \nabla \fmbf(\xmbf_{k})\right\|_2, \numberthis \label{eq: y_bar_bnd_twoterm}
\end{align*}
where the equality follows since %is obtained by adding 
$-\Imbf(\ymbf_k - \ybb_k) = 0$, and the inequality is due to the  triangle inequality.  
The term $\left\|\nabla \fmbf(\xmbf_{k+1}) - \nabla \fmbf(\xmbf_{k}) \right\|_2$ can be bounded as follows,
\begin{align*}
&\left\|\nabla \fmbf(\xmbf_{k+1}) - \nabla \fmbf(\xmbf_{k}) \right\|_2  \\
\leq& L\|\xmbf_{k+1} - \xmbf_{k}\|_2 \\
=& L\|\Rmbf_1^{n_c}\xmbf_{k} - \alpha \Rmbf_2^{n_c}\ymbf_k - \xmbf_{k}\|_2   \\
=& L\|(\Rmbf_1^{n_c}-I_{nd})\xmbf_{k} - (\Rmbf_1^{n_c} - I_{nd})\Bar{\xmbf}_k  - \alpha \Rmbf_2^{n_c}\ymbf_k\|_2 \\
=& L\|(\Rmbf_1^{n_c}-I_{nd})(\xmbf_{k} - \Bar{\xmbf}_k) - \alpha \Rmbf_2^{n_c}\ymbf_k\|_2 \\
\leq& L\|\Rmbf_1^{n_c}-I_{nd}\|_2\|\xmbf_{k} - \Bar{\xmbf}_k\|_2 + \alpha L\|\Rmbf_2^{n_c}\|_2\|\ymbf_k + \Bar{\ymbf}_k - \Bar{\ymbf}_k\|_2   \\
\leq& L\|\Rmbf_1^{n_c}-I_{nd}\|_2\|\xmbf_{k} - \Bar{\xmbf}_k\|_2 + \alpha L\|\Rmbf_2^{n_c}\|_2\|\ymbf_k - \Bar{\ymbf}_k\|_2 + \alpha L\|\Rmbf_2^{n_c}\|_2\left\|\Bar{\ymbf}_k\right\|_2, \numberthis \label{eq: grad_diff_bnd}
\end{align*}
where the first inequality is due to \cref{asum: convex_and_smooth}, the first equality is due to the iterate update form in \eqref{eq: iter_form_simple}, the second equality is due to the fact that $-(\Rmbf_1^{n_c} - I_{nd})\xbb_k = 0$, and the last two inequalities are due to applications of the triangle inequality.
Next, we bound the term $\|\ybb_k\|_2$. By \eqref{eq: derivative_terms_define}, \cref{asum: convex_and_smooth} and $\sum_{i=1}^n \|x_{i, k} - \xbar_k\|_2 \leq \sqrt{n}\|\xmbf_{k} - \Bar{\xmbf}_{k}\|_2$,
\begin{align*}
\left\|\Bar{\ymbf}_k\right\|_2 & \leq \sqrt{n} \|\bar{y}_k\|_2 = \sqrt{n} \|g_k\|_2 \\
&\leq \sqrt{n}\left\|\frac{1}{n} \sum_{i = 1}^n \nabla f_i(x_{i, k}) -  \frac{1}{n} \sum_{i = 1}^n \nabla f_i(\bar{x}_{k})\right\|_2 + \sqrt{n}\left\|\frac{1}{n} \sum_{i = 1}^n \nabla f_i(\bar{x}_{k})\right\|_2  \\
&= \frac{1}{\sqrt{n}}\left\| \sum_{i = 1}^n \nabla f_i(x_{i, k}) -  \sum_{i = 1}^n \nabla f_i(\bar{x}_{k})\right\|_2 + \frac{1}{\sqrt{n}} \left\| \sum_{i = 1}^n \nabla f_i(\bar{x}_{k}) - \sum_{i = 1}^n \nabla f_i(x^*)\right\|_2 \\
% &\leq \frac{1}{\sqrt{n}}\left\|\nabla f(x_{k}) - \nabla f(\bar{x}_k)\right\|_2 + \frac{1}{\sqrt{n}}\left\|\nabla f(\bar{x}_k) - \nabla f((x^*)^T)\right\|_2 \\
&\leq L\left\|\xmbf_k - \bar{\xmbf}_k\right\|_2 + L\sqrt{n} \| \Bar{x}_{k} - x^*\|_2 \numberthis \label{eq : y_bar_bound}.
\end{align*}
Thus, by \eqref{eq: y_bar_bnd_twoterm}, \eqref{eq: grad_diff_bnd} and \eqref{eq : y_bar_bound},
\begin{align*} 
    \left\|\ymbf_{k+1} - \Bar{\ymbf}_{k+1}\right\|_2 
     \leq &\left( \left\|\Rmbf_3^{n_c} - \Imbf\right\|_2 + \alpha L \left\|\Rmbf_4^{n_c} - \Imbf\right\|_2 \|\Rmbf_2^{n_c}\|_2\right)\left\|\ymbf_k - \Bar{\ymbf}_k\right\|_2 \\
    & + \left( L \left\|\Rmbf_4^{n_c} - \Imbf\right\|_2 \|\Rmbf_1^{n_c}-I_{nd}\|_2 + \alpha L^2\left\|\Rmbf_4^{n_c} - \Imbf\right\|_2 \|\Rmbf_2^{n_c}\|_2\right) \|\xmbf_{k} - \Bar{\xmbf}_k\|_2 \\
    &  + \alpha L^2 \left\|\Rmbf_4^{n_c} - \Imbf\right\|_2 \|\Rmbf_2^{n_c}\|_2 \sqrt{n} \| \Bar{x}_{k} - x^*\|_2. 
\end{align*}
The conditional expectation given $\xmbf_k$ and $\ymbf_k$ of the consensus error above can be bounded as, 
\begin{align*} 
    &\E\left[\left\|\ymbf_{k+1} - \Bar{\ymbf}_{k+1} \right\|_2 | \xmbf_k, \ymbf_k\right] \\ 
     \leq & \E\left[ \left\|\Rmbf_3^{n_c} - \Imbf\right\|_2 + \alpha L \left\|\Rmbf_4^{n_c} - \Imbf\right\|_2 \|\Rmbf_2^{n_c}\|_2 | \xmbf_k, \ymbf_k \right] \left\|\ymbf_k - \Bar{\ymbf}_k\right\|_2 \numberthis \label{eq: y_con_before_expectation}\\
    & + \E\left[ L \left\|\Rmbf_4^{n_c} - \Imbf\right\|_2 \|\Rmbf_1^{n_c}-I_{nd}\|_2 + \alpha L^2\left\|\Rmbf_4^{n_c} - \Imbf\right\|_2 \|\Rmbf_2^{n_c}\|_2 | \xmbf_k, \ymbf_k \right] \|\xmbf_{k} - \Bar{\xmbf}_k\|_2 \\
    &  + \alpha L^2 \E\left[\left\|\Rmbf_4^{n_c} - \Imbf\right\|_2 \|\Rmbf_2^{n_c}\|_2 | \xmbf_k, \ymbf_k \right] \sqrt{n} \| \Bar{x}_{k} - x^*\|_2.  
\end{align*}
The matrices $\Rmbf_1$, $\Rmbf_2$, $\Rmbf_3$ and $\Rmbf_4$ are independent of $\xmbf_k$ and $\ymbf_k$, and as such the expressions above can be simplified using
\begin{align*}
    \E[ \left\|\Rmbf_4^{n_c} - \Imbf\right\|_2 \|\Rmbf_2^{n_c}\|_2 | \xmbf_k, \ymbf_k ] &=  p \beta_4^{n_c} \|Z_2^{n_c}\|_2 + (1-p) \left\|I_{nd} - \Imbf\right\|_2 \|I_{nd}\|_2 = \eta_4, \\
    \E[ \left\|\Rmbf_4^{n_c} - \Imbf\right\|_2 \|\Rmbf_1^{n_c}-I_{nd}\|_2 | \xmbf_k, \ymbf_k ] &=  p\beta_4^{n_c} \|\Zmbf_1^{n_c}-I_{nd}\|_2 + (1-p)\left\|I_{nd} - \Imbf\right\|_2 \|I_{nd} - I_{nd}\|_2 \\
    & = p\beta_4^{n_c} \|\Zmbf_1^{n_c}-I_{nd}\|_2.
\end{align*}
%After 
Substituting \eqref{eq: R_expected_norm} and the above expressions in \eqref{eq: y_con_before_expectation}, and taking the total expectation of the above (with respect to all the random variables starting with $\xmbf_0, \ymbf_0$) yields, 
% a full expectation, % yields, %over it yields
\begin{align*}  
    \E[\left\|\ymbf_{k+1} - \Bar{\ymbf}_{k+1}\right\|_2 ] 
    & \leq \alpha L^2 \eta_4 \sqrt{n} \E[\| \Bar{x}_{k} - x^*\|_2] + \left( \eta_3 + \alpha L \eta_4 \right) \E[\left\|\ymbf_k - \Bar{\ymbf}_k\right\|_2] \numberthis \label{eq: con_y_bound} \\
    &\quad + \left( p\beta_4^{n_c} L\|\Zmbf_1^{n_c}-I_{nd}\|_2 + \alpha L^2\eta_4 \right) \E[\|\xmbf_{k} - \Bar{\xmbf}_k\|_2].
\end{align*}
Combining \eqref{eq : opt_bound}, \eqref{eq: con_x_bound} and \eqref{eq: con_y_bound} yields the desired form in \eqref{eq: general_r}--\eqref{eq: general_A}.
\eproof

Using \cref{lem: lyapunov}, we now provide the explicit form of the matrix $A(n_c, p)$ for the methods described in \cref{tab: Algorithm_Def}.% to establish the progression of the expected error vector $r_k$.

\begin{corollary} \label{col: lypanov}
Suppose \cref{asum: convex_and_smooth} holds and $\alpha \leq \frac{1}{L}$.
%Suppose the conditions of \cref{lem: lyapunov} are satisfied. 
Then, the matrices $A({n_c, p})$ for the methods described in \cref{tab: Algorithm_Def} are defined as:
\begin{align*}
    \mbox{\texttt{RGTA-1: }} & \quad A_1({n_c, p}) = \begin{bmatrix}
            1 - \alpha \mu & \frac{\alpha L}{\sqrt{n}} & 0\\
            0 & \eta & \alpha\\
            \sqrt{n}\alpha L^2 & L(2p + \alpha L) & \eta + \alpha L\\
        \end{bmatrix},\\
    \mbox{\texttt{RGTA-2: }} & \quad A_2({n_c, p})  = \begin{bmatrix}
            1 - \alpha \mu & \frac{\alpha L}{\sqrt{n}} & 0\\
            0 & \eta & \alpha \eta\\
            \sqrt{n}\alpha L^2 & L(2p + \alpha L) & \eta + \alpha L\\
        \end{bmatrix},   \numberthis \label{eq: A_matrices}\\
    \mbox{\texttt{RGTA-3: }} & \quad A_3({n_c, p}) = \begin{bmatrix}
            1 - \alpha \mu & \frac{\alpha L}{\sqrt{n}} & 0\\
            0 & \eta & \alpha \eta\\
            \eta \sqrt{n}\alpha L^2 & L(2\beta^{n_c}p + \eta\alpha L) & \eta(1 + \alpha L)\\
        \end{bmatrix},
\end{align*}
where $\eta = p\beta^{n_c} + (1-p)$.
\end{corollary}
\begin{proof}
Substituting the communication matrices for each method in \eqref{eq: general_A} and using $\|\Zmbf_1^{n_c} - I_{nd}\|_2 \leq 2$ gives the desired result. 
\end{proof}

We analyze the convergence properties of \texttt{RGTA} using the spectral radius of the matrix $A(n_c, p)$, i.e., $\rho(A(n_c, p))$, as it presents an upper bound on the matrix norm. We first qualitatively establish the effect of the number of communication steps ($n_c$), 
% communication ($n_c$)
and the relative performance of \texttt{RGTA-1}, \texttt{RGTA-2} and \texttt{RGTA-3}.

\btheorem   \label{th.incr_rates}
Suppose \cref{asum: convex_and_smooth} holds and $\alpha \leq \frac{1}{L}$ in \cref{alg: Randomised}. Then, with an increase in $n_c$, $\rho(A(n_c, p))$ decreases, where $A(n_c, p)$ is defined in  \cref{lem: lyapunov}. Thus, as $n_c$ increases, $\rho(A_i(n_c, p))$ decreases for all $i =1, 2, 3$ defined in \cref{eq: A_matrices}. Moreover, if all three methods defined in \cref{tab: Algorithm_Def} (\texttt{RGTA-1}, \texttt{RGTA-2} and \texttt{RGTA-3}) employ the same step size, then, $\rho(A_1({n_c, p})) \geq \rho(A_2({n_c, p})) \geq \rho(A_3({n_c, p}))$.
\etheorem

\bproof
Note that $A(n_c, p) \geq 0$ and $A(n_c, p) \geq A(n_c + 1, p)$.  By \cite[Corollary 8.1.19]{horn2012matrix}, it follows that $\rho(A(n_c, p)) \geq \rho(A(n_c + 1, p))$. The same argument is applicable for $A_1({n_c, p})$, $A_2({n_c, p})$ and $A_3({n_c, p})$. Next, observe that $A_1({n_c, p}) \geq A_2({n_c, p}) \geq A_3({n_c, p}) \geq 0$ when the same step size is employed. Thus, again by \cite[Corollary 8.1.19]{horn2012matrix}, it follows that $\rho(A_1({n_c, p})) \geq \rho(A_2({n_c, p})) \geq \rho(A_3({n_c, p}))$.
\eproof

While \cref{th.incr_rates} states that there is a monotonic decrease in $\rho(A(n_c, p))$ with an increase in $n_c$, the same argument cannot be applied to analyze %used for 
the effect of the communication probability ($p$). This is because the effect of $p$ on the term $[A(n_c, p)]_{3, 2} = p\beta_4^{n_c}L\|\Zmbf_1^{n_c}-I_{nd}\|_2 + \alpha L^2 \eta_4$ depends on the choice of the communication matrices.  
For instance, %example, 
in $A_1(n_c, p)$ and $A_2(n_c, p)$ from \cref{col: lypanov}, the term $L(2p + \alpha L)$ increases with an increase in $p$. In contrast, 
%, while 
in $A_3(n_c, p)$, the effect of $p$ on the term in the same position in the matrix ($L(2\beta^{n_c}p + \eta\alpha L)$) depends on the network and problem parameters.

%Now we 
Next, we present conditions under which a linear rate of convergence can be derived in expectation in terms of the network parameters ($\beta_1$, $\beta_2$, $\beta_3$, $\beta_4$) and problem parameters ($\mu$, $L$ and $\kappa = \frac{L}{\mu}$) for \texttt{RGTA}.
\btheorem \label{th. general_step_cond}
 Suppose \cref{asum: convex_and_smooth} holds, $0 < p < 1$ and $n_c \geq 1$ in \cref{alg: Randomised}. If 
 \begin{align} \label{eq: gen_step_cond}
    \alpha < min \left\{\tfrac{1}{L}, \tfrac{1 - \eta_3}{L\eta_4} , \tfrac{\eta_4(1 - \eta_1) + 2\eta_2p\beta_4^{n_c}}{2\eta_2\eta_4\kappa(L + \mu)} \left[\sqrt{1 + \tfrac{4(1-\eta_1)(1-\eta_3)\eta_2\eta_4(\kappa+1)}{(\eta_4(1 - \eta_1) + 2\eta_2p\beta_4^{n_c})^2}} - 1 \right] \right\}, 
 \end{align}
$\beta_1, \beta_3 < 1$, and $\eta_i = p\beta_i^{n_c} + (1 - p)$ for all $i = 1, 2, 3, 4$, then, for all $\delta > 0$ there exists a constant $C_{\delta}>0$ such that, for all $k \geq 0$,
\begin{align*}
    \|r_{k}\|_2 \leq C_{\delta}(\rho(A({n_c, p})) + \delta)^k \|r_0\|_2, \quad \mbox{where } \rho(A({n_c, p})) < 1.
\end{align*}
\etheorem

\bproof
Following \cite[Lemma 5]{pu2021distributed} that uses the Perron-Forbenius Theorem (\cite[Theorem 8.4.4]{horn2012matrix}), when the $3\times3$ matrix $A(n_c, p)$ is nonnegative and irreducible, it is sufficient to show that the diagonal elements of $A(n_c, p)$ are less than one and $det(I_3 - A({n_c, p})) > 0$ in order to guarantee $\rho(A(n_c, p)) < 1$. If $p<1$, the matrix $A(n_c, p)$ is always irreducible as $\eta_i > 0$ for all $i = 1, 2, 3, 4$.

We first consider the diagonal elements of the matrix $A(n_c, p)$. The first element is $1 - \alpha \mu \leq 1 -\frac{\mu}{L} < 1$
due to \eqref{eq: gen_step_cond}. The second element is $\eta_1 < 1$, as $\beta_1 < 1$ and $p > 0$. Finally, the third element is 
$\eta_3 + \alpha\eta_4L < \eta_3 + \frac{1 - \eta_3}{\eta_4 L}\eta_4 L = 1$
due to \eqref{eq: gen_step_cond} and $\eta_3 < 1$ as $\beta_3 < 1$ and $ p > 0$.

Next, we consider
\begin{align*}
&\det(I_3 - A(n_c, p)) \\
=&-\alpha \left(\alpha^2\left(L^3\eta_2\eta_4 + \mu L^2\eta_2\eta_4\right) + \alpha\left(L\mu\eta_4(1- \eta_1) + 2L\mu\eta_2 p \beta_4^{n_c}\right) - \mu\left(1 - \eta_1\right)\left(1 - \eta_3\right)\right)  \\
=&-\left(L + \mu \right) L^2 \eta_2\eta_4 \alpha(\alpha - \alpha_l)(\alpha - \alpha_u),
\end{align*}
where $\alpha_l = \alpha_1 - \alpha_2$, $\alpha_u = \alpha_1 + \alpha_2$, and
\begin{align*}
    \alpha_1 = -\tfrac{\eta_4(1 - \eta_1) + 2\eta_2p\beta_4^{n_c}}{2\eta_2\eta_4\kappa(L + \mu)} \quad \mbox{and} \quad \alpha_2 = -\alpha_1 \sqrt{1 + \tfrac{4(1-\eta_1)(1-\eta_3)\eta_2\eta_4(\kappa+1)}{(\eta_4(1 - \eta_1) + 2\eta_2p\beta_4^{n_c})^2}} .
\end{align*}
Observe that $\alpha_l < 0 < \alpha_u$ since $\alpha_1 < 0$ and $\alpha_2 > |\alpha_1|$. From \eqref{eq: gen_step_cond}, we have $0 < \alpha < \alpha_u$. Therefore, $\det(I_3 - A(n_c, p)) > 0$, which combined with the fact that the diagonal elements of the matrix are less than 1, implies $\rho(A(n_c, p)) < 1$.

Finally, we bound the norm of error vector $\|r_k\|_2$ by telescoping $r_{i+1} \leq A(n_c, p) r_{i}$ from $i = 0$ to $k-1$ and the triangle inequality,
\begin{align*}
    \|r_{k}\|_2 &\leq \|A({n_c, p})^k\|_2\|r_0\|_2.
\end{align*}
By \cite[Corollary 5.6.13]{horn2012matrix}, we can bound $\|A({n_c, p})^k\|_2 \leq C_{\delta}(\rho(A(n_c, p)) + \delta)^k$, where $\delta > 0$ and $C_{\delta}$ is a positive constant that depends %depending 
on $A(n_c, p)$ and $\delta$.
\eproof

\cref{th. general_step_cond} shows that a sufficient condition to ensure a linear rate of convergence in expectation is that the matrices $\Wmbf_1$ and $\Wmbf_3$ represent connected networks.
This does not impose any relationship between the communication matrices $\Wmbf_1$, $\Wmbf_2$, $\Wmbf_3$ and $\Wmbf_4$. These are the same minimal conditions on the network established for the iterate update form \eqref{eq: general_form} in \cite{berahas2023balancing}. 

In the next %following 
corollary, we establish explicit conditions for a linear rate of convergence in expectation for \texttt{RGTA-1}, \texttt{RGTA-2} and \texttt{RGTA-3}. 

\bcorollary \label{col. step_cond}
Suppose \cref{asum: convex_and_smooth} holds, $0 < p < 1$ and $n_c \geq 1$ in \cref{alg: Randomised}. If the following step size conditions hold for the methods described in \cref{tab: Algorithm_Def},
\begin{equation} \label{eq: step_cond_methods}
\begin{aligned}
    \mbox{\texttt{RGTA-1:}}       
    \quad &\alpha < min \left\{ \tfrac{1 - \eta}{L} , p\tfrac{(3 - \beta^{n_c})}{2\kappa(L + \mu)}\left[\sqrt{1 + 4(\kappa + 1)\left[ \tfrac{1 - \beta^{n_c}}{3 - \beta^{n_c}} \right]^2} - 1\right] \right\}, \\
    \mbox{\texttt{RGTA-2:}}   
    \quad &\alpha < min \left\{\tfrac{1 - \eta}{L} , p\tfrac{(2\eta + 1 - \beta^{n_c})}{2\kappa\eta(L + \mu)}\left[\sqrt{1 + 4\eta(\kappa + 1)\left[ \tfrac{1 - \beta^{n_c}}{2\eta + 1 - \beta^{n_c}} \right]^2} - 1\right]  \right\}, \\
    \mbox{\texttt{RGTA-3:}}    
    \quad &\alpha < min \left\{\tfrac{1}{L}, \tfrac{1 - \eta}{L\eta} , p\tfrac{1 + \beta^{n_c}}{2\kappa\eta(L + \mu)}\left[\sqrt{1 + 4(\kappa + 1)\left[ \tfrac{1 - \beta^{n_c}}{1 + \beta^{n_c}} \right]^2} - 1\right]\right\},   
\end{aligned}
\end{equation}
where $\eta = p\beta^{n_c} + (1-p)$, then, for all $\delta > 0$ there exists a constant $C_{i, \delta}>0$ such that, for all $k \geq 0$,
\begin{align*}
    \|r_{k}\|_2 \leq C_{i, \delta}(\rho(A_i({n_c, p})) + \delta)^k \|r_0\|_2, \quad \mbox{where $\rho(A_i({n_c, p})) < 1, \quad \forall \,\, i=1, 2, 3$.}
\end{align*}
\ecorollary
\bproof
The conditions given in \cref{th. general_step_cond} are satisfied by all three methods. That is, the matrices $A_1(n_c, p)$, $A_2(n_c, p)$, and $A_3(n_c, p)$ are irreducible as $\eta > 0$ because $ p < 1$, and $\beta < 1$. 
% $\eta < 1$ because $ p > 0$ and $\beta < 1$. 
Therefore, substituting values for communication matrices in \eqref{eq: gen_step_cond} yields the desired result. Since $p < 1$, $\eta > 0$ and  $\frac{1-\eta}{L} < \frac{1}{L}$, the \texttt{RGTA-1} and  \texttt{RGTA-2} step size bounds do not feature $\frac{1}{L}$.
\eproof
 
The step size conditions in \cref{col. step_cond} show how performing more communications, i.e., increasing $n_c$ or $p$ allows for a large step size to be employed for all three methods. That said, the effect is the strongest in \texttt{RGTA-3} due to the presence of $\eta$ in the denominator of the second term. This allows for the possibility of employing a step size similar to gradient descent, i.e., $\frac{1}{L}$, %gradient %decent-like 
%\rb{descent like} step size of $\frac{1}{L}$ 
when enough communication steps are performed (high $n_c$) sufficiently frequently (high $p$). The effect is comparatively weaker in \texttt{RGTA-2} and the weakest in \texttt{RGTA-1}, i.e., under the same parameters ($p$ and $n_c$), a larger step size is potentially permissible in \texttt{RGTA-2} than in \texttt{RGTA-1}.

In \cref{th. general_step_cond}, we derived a sufficient condition on the step size such that \cref{alg: Randomised} converges at a linear rate. We now provide an upper bound on the convergence rate for a sufficiently small step size $\alpha >0$. 

\btheorem  \label{th. general_rate_bound}
    Suppose \cref{asum: convex_and_smooth} holds, $0 < p < 1$ and $n_c \geq 1$ in \cref{alg: Randomised}. If $\alpha \leq \frac{1}{L}$, then, 
    \begin{align*}
        \rho(A(n_c, p)) \leq \lambda_u = \max\left\{1 - \tfrac{\alpha\mu}{2}, \hat{\lambda} + \sqrt{2\alpha L \kappa \eta_2\eta_4}\right\}, \numberthis \label{eq: general_rate_rate}
    \end{align*}
    where $\hat{\lambda} = \tfrac{\eta_1 + \eta_3 + \alpha L\eta_4  + \sqrt{\left(\eta_1 - \eta_3 - \alpha L \eta_4 \right)^2  + 4\eta_2\eta_4 \alpha^2 L^2 + 8p\alpha L\eta_2\beta_4^{n_c}}}{2} $ 
    % \begin{align*}
    % \hat{\lambda} = \tfrac{\eta_1 + \eta_3 + L\alpha\eta_4  + \sqrt{\left(\eta_1 - \eta_3 - \alpha L \eta_4 \right)^2  + 4\eta_2\eta_4 L^2\alpha^2 + 8p\alpha L\eta_2\beta_4^{n_c}}}{2} 
    % \end{align*}
    and $\eta_i = p\beta_i^{n_c} + (1 - p)$ for all $i = 1, 2, 3, 4$.
\etheorem

\bproof
The matrix $A(n_c, p)$ is nonnegative %non-negative 
and irreducible under $0 < p < 1$ as $\eta_i > 0$ for all $i=1, 2, 3, 4$. Thus, by the Perron-Forbenius Theorem (\cite[Theorem 8.4.4]{horn2012matrix}), the spectral radius of the matrix $A(n_c, p)$ is equal to its largest eigenvalue which is a positive real number. We prove that $\lambda_u$ is an upper bound on the largest eigenvalue following a similar approach to \cite{qu2017harnessing}. We do this by showing the characteristic equation is positive at $\lambda_u$ and increasing for all values greater than $\lambda_u$. Consider, 
 \begin{align*}
g(\lambda) &= \det(\lambda I_3 - A(n_c, p)) \\
&= (\lambda - 1 + \alpha\mu)\left((\lambda - \eta_1)(\lambda - \eta_3 - \alpha L\eta_4) - \alpha L\eta_2(2p\beta_4^{n_c} + \alpha L\eta_4)\right) - \alpha^3 L^3 \eta_2\eta_4 \\
&= (\lambda - 1 + \alpha\mu)q(\lambda) - \alpha^3 L^3 \eta_2\eta_4 \numberthis \label{eq. chareq}
 \end{align*}
where $q(\lambda) = \lambda^2 -\lambda(\eta_1 + \eta_3 + \alpha L\eta_4)  + \eta_1(\eta_3 + \alpha L\eta_4) - \alpha L\eta_2(2p\beta_4^{n_c} + \alpha L\eta_4)$. Let the roots of the quadratic function $q(\lambda)$ be denoted as $\lambda_1$ and $\lambda_2$. Then, we have,
\begin{align*}
\max\{\lambda_1, \lambda_2\} &= \tfrac{\eta_1 + \eta_3 + \alpha L\eta_4  + \sqrt{\left(\eta_1 + \eta_3 + \alpha L\eta_4\right)^2 - 4\left(\eta_1(\eta_3 + \alpha L\eta_4) - \alpha L\eta_2(2p\beta_4^{n_c} + \alpha L\eta_4)\right)}}{2}\\
&= \tfrac{\eta_1 + \eta_3 + \alpha L \eta_4  + \sqrt{\left(\eta_1 - \eta_3 - \alpha L \eta_4 \right)^2  + 4\eta_2\eta_4 \alpha^2L^2 + 8p\alpha L\eta_2\beta_4^{n_c}}}{2}. 
\end{align*}
Therefore, for any $\lambda \geq \max\left\{1 - \alpha \mu, \hat{\lambda}\right\}$, $g(\lambda)$ is an increasing function and lower bounded by $(\lambda - 1 + \alpha\mu)(\lambda - \hat{\lambda})^2 - \alpha^3 L^3 \eta_2\eta_4$. 
Thus,
\begin{align*}
g(\lambda_u) &\geq (\lambda - 1 + \alpha\mu)(\lambda - \hat{\lambda})^2 - \alpha^3 L^3 \eta_2\eta_4 \\
&\geq \left(1 - \frac{\alpha\mu}{2} -1 + \alpha\mu\right)(\lambda - \hat{\lambda})^2 - \alpha^3 L^3 \eta_2\eta_4 \\
&\geq \frac{\alpha \mu}{2}\left(\frac{2\alpha L^2\eta_2\eta_4}{\mu}\right)  - \alpha^3 L^3 \eta_2\eta_4 \\
&=\alpha^2L^2\eta_2\eta_4(1 - \alpha L) \geq 0,
\end{align*}
where the second and third inequalities are due to the definition of $\lambda_u$ and the final quantity is nonnegative due to $\alpha L \leq 1$. From the above arguments, we can conclude that $\rho(A({n_c, p})) \leq \lambda_u$ which completes the proof. 
\eproof

\cref{th. general_rate_bound} is derived independent of the conditions in \cref{th. general_step_cond}. When $\rho(A(n_c, p)) < 1$ is imposed as a condition using \cref{th. general_rate_bound}, $\beta_1, \beta_3 < 1$ arises %is recognised 
as a necessary condition for convergence. We show this requirement by constructing a lower bound on $\lambda_u$ as  $\lambda_u \geq \hat{\lambda} \geq \tfrac{\eta_1 + \eta_3 + |\eta_1 - \eta_3|}{2}$.
For convergence we require $\lambda_u < 1$, i.e., $\tfrac{\eta_1 + \eta_3 + |\eta_1 - \eta_3|}{2} < 1$ which implies  $\eta_1, \eta_3 < 1$ and, consequently, 
%thus 
$\beta_1, \beta_3 < 1$. Thus, we again require $\Wmbf_1$ and $\Wmbf_3$ to represent connected networks. The step size condition in \cref{th. general_step_cond} is $\mathcal{O}(L^{-1}\kappa^{-1/2})$ while the step size condition from \cref{th. general_rate_bound} is $\mathcal{O}(L^{-1}\kappa^{-1})$, which is more pessimistic. 
We use \cref{th. general_rate_bound} to give an explicit bound on the convergence rate which allows us to better differentiate among different methods and the effect of the  parameters $n_c$ and $p$, see \cref{th. rate_bound}.

\bcorollary  \label{th. rate_bound}
Suppose \cref{asum: convex_and_smooth} holds, $0 < p < 1$ and $n_c \geq 1$ in \cref{alg: Randomised}. If $\alpha \leq \frac{1}{L}$, then the spectral radii for the methods described in \cref{tab: Algorithm_Def} satisfy
    \begin{align*}
        \mbox{\texttt{RGTA-1:}}       \quad 
            \rho(A_1({n_c, p}))  & < \max\left\{1 - \tfrac{\alpha\mu}{2}, \eta + \sqrt{\alpha L} \left(2.5 + \sqrt{2\kappa}\right)\right\} \\
        \mbox{\texttt{RGTA-2:}}   \quad 
            \rho(A_2({n_c, p})) & < \max\left\{1 - \tfrac{\alpha\mu}{2}, \eta + \sqrt{\alpha L} \left(2.5 + \sqrt{2\kappa \eta}\right)\right\} \numberthis \label{eq: rates} \\
        \mbox{\texttt{RGTA-3:}}    \quad 
            \rho(A_3({n_c, p})) & < \max\left\{1 - \tfrac{\alpha\mu}{2}, \eta\left(1 + \sqrt{\alpha L} \left(2.5 + \sqrt{2\kappa}\right)\right)\right\}
    \end{align*}
where $\eta = p\beta^{n_c} + (1 - p)$.
\ecorollary

\bproof
Since the conditions in \cref{th. general_rate_bound} are satisfied, we can plug in the values for $\beta_i$ and $\eta_i$, for all $i=1,2,3,4$, for each method to get an upper bound on the spectral radius. The upper bound $\lambda_u$ for \texttt{RGTA-1} can be simplified as,
\begin{align*}
    \hat{\lambda} + \sqrt{\frac{2\alpha L^2 \eta_2\eta_4}{\mu}} &= \frac{2\eta + \alpha L  + \sqrt{5\alpha^2L^2 + 8p\alpha L}}{2} + \sqrt{\frac{2\alpha L^2}{\mu}} \\
    &= \eta + \frac{\sqrt{\alpha L}}{2}\left(\sqrt{\alpha L} + 2\sqrt{2\kappa} + \sqrt{8p + 5\alpha L} \right) \\
    &\leq \eta + \sqrt{\alpha L} \left(2.5 + \sqrt{2\kappa}\right),
\end{align*}
where the last inequality is due to  $\alpha \leq \frac{1}{L}$ and $p < 1$. Following the same approach for \texttt{RGTA-2}, $\lambda_u$  can be simplified as,
\begin{align*}
    \hat{\lambda} + \sqrt{\frac{2\alpha L^2 \eta_2\eta_4}{\mu}} &= \frac{2\eta + \alpha L  + \sqrt{\alpha^2 L^2 + 4\alpha^2 L^2\eta + 8p\alpha L\eta}}{2} + \sqrt{\frac{2\alpha L^2\eta}{\mu}} \\
    &= \eta + \frac{\sqrt{\alpha L}}{2}\left(\sqrt{\alpha L} + 2\sqrt{2\kappa \eta} + \sqrt{8p\eta + 4\alpha L\eta + \alpha L} \right) \\
    &\leq \eta + \sqrt{\alpha L} \left(2.5 + \sqrt{2\kappa \eta}\right),
\end{align*}
where the last inequality is due to $\alpha \leq \frac{1}{L}$ and $p < 1$. Finally, the upper bound for \texttt{RGTA-3} can be simplified as, %is \rb{given as,}
\begin{align*}
    \hat{\lambda} + \sqrt{\frac{2\alpha L^2 \eta_2\eta_4}{\mu}} &= \frac{2\eta + \alpha L\eta  + \sqrt{ 5\alpha^2 L^2(\eta)^2 + 8p\alpha L(\eta\beta^{n_c})}}{2} + \sqrt{\frac{2\alpha L^2(\eta)^2}{\mu}} \\
    &\leq \eta\left(1  + \frac{\sqrt{\alpha L}}{2}\left(\sqrt{\alpha L} + 2\sqrt{2\kappa} + \sqrt{8p + 5\alpha L} \right)\right) \\
    &\leq \eta\left(1 + \sqrt{\alpha L} \left(2.5 + \sqrt{2\kappa}\right)\right),
\end{align*}
where the first inequality is due to $\eta \geq \beta^{n_c}$ and the last due to $\alpha \leq \frac{1}{L}$ and $ p < 1$.
\eproof

By \cref{th. rate_bound},  it is evident that the %we can see the 
convergence rate for all three methods improves with an increase in communication steps, i.e., increase in $n_c$ or $p$. When %While 
deriving these bounds, we upper bound the probability $p$ by one in the last summand. 
While that summand shows worsening of the convergence rate with an increase in $p$, as it is dominated by $\sqrt{\kappa}$, bounding it does not create any contradictions to our inferences.

\cref{th. general_step_cond,th. general_rate_bound} 
%Theorems \ref{th. general_step_cond} and \ref{th. general_rate_bound} 
assume that $0 < p < 1$ for the general analysis of \cref{alg: Randomised}. While $ p > 0$ is a required condition to guarantee convergence of the iterates, 
%\sg{for convergence}, 
$p < 1$ is only required to ensure that the matrix $A(n_c, p)$ is irreducible. When $p=1$, i.e., communication is performed every iteration, the matrix $A(n_c, p)$ can be reducible for certain choices of communication matrices. For example, when applied to fully connected networks with equal weights on all nodes, i.e., $\Wmbf = \tfrac{1_n1_n^T}{n}$, resulting in $\beta = 0$, with $p=1$, the matrices in both \texttt{RGTA-2} ($A_2(n_c, 1)$) and \texttt{RGTA-3} ($A_3(n_c, 1)$) become reducible. 
%For example, the matrices in \texttt{RGTA-2} and \texttt{RGTA-3} \sg{with $p=1$} when applied to fully connected networks \sg{with equal weights on nodes, i.e., $\Wmbf = \tfrac{1_n1_n^T}{n}$, thus $\beta = 0$}. 
Analyzing these cases involves working with a reduced linear system compared to \eqref{eq: general_r} since some components of of $r_k$ are always zero, e.g.,  $\|\xmbf_k - \xbb_k\|_2$.
%Analysing such cases uses a reduced linear system compared to \eqref{eq: general_r}. 
%This is because some components of $r_k$ are always zero, e.g.,  $\|\xmbf_k - \xbb_k\|_2$, in these cases.
The general analysis with $p=1$ can be found in %has been performed in 
\cite[Section 3.1]{berahas2023balancing}. When $\beta > 0$ and $p=1$, it follows that
\begin{align*}
    \mbox{\texttt{RGTA-1:}}       \quad 
        \rho(A_1({n_c, 1}))  & < \max\left\{1 - \tfrac{\alpha\mu}{2}, \beta^{n_c} + \sqrt{\alpha L} \left(2.5 + \sqrt{2\kappa}\right)\right\} \\
    \mbox{\texttt{RGTA-2:}}   \quad 
        \rho(A_2({n_c, 1})) & < \max\left\{1 - \tfrac{\alpha\mu}{2},  \beta^{n_c} + \sqrt{\alpha L} \left(2.5 + \sqrt{2\kappa \beta^{n_c}}\right)\right\}  \\
    \mbox{\texttt{RGTA-3:}}    \quad 
        \rho(A_3({n_c, 1})) & < \max\left\{1 - \tfrac{\alpha\mu}{2}, \beta^{n_c}\left(1 + \sqrt{\alpha L} \left(2.5 + \sqrt{2\kappa}\right)\right)\right\}.
\end{align*}
The convergence rate, in the case of $p=1$, still improves with an increase in the number of communication steps ($n_c$) and the relative ordering among methods discussed earlier %above 
remains the same. In addition, when $p=1$ and the network is fully connected with equal weights ($\beta = 0$), \texttt{RGTA-2} and \texttt{RGTA-3} exhibit performance similar to that of gradient descent with a $(1-\alpha\mu/2)$ convergence rate; see \cite[Section 3.3]{berahas2023balancing}.
%\rb{a} gradient descent like performance with $(1-\alpha\mu/2)$ convergence rate.

%%%%%%%%%%%%%%%%%%%%%%%%%%%%%%%%%%%%%%%%%%%%%%%%%%%%%%%%%%%%%%%%%%%%%%%%%%%%%%%%
%%%%%%%%%%%%%%%%%%%%%%%%%%%%%%%%%%%%%%%%%%%%%%%%%%%%%%%%%%%%%%%%%%%%%%%%%%%%%%%%
%%%%%%%%%%%%%%%%%%%%%%%%%%%%%%%%%%%%%%%%%%%%%%%%%%%%%%%%%%%%%%%%%%%%%%%%%%%%%%%%
%%%%%%%%%%%%%%%%%%%%%%%%%%%%%%%%%%%%%%%%%%%%%%%%%%%%%%%%%%%%%%%%%%%%%%%%%%%%%%%%
%%%%%%%%%%%%%%%%%%%%%%%%%%%%%%%%%%%%%%%%%%%%%%%%%%%%%%%%%%%%%%%%%%%%%%%%%%%%%%%%
\subsection{Computation and Communication Complexity}\label{sec.compl}
We now analyze the computation and communication complexity of the \texttt{RGTA} framework, i.e., the number of computation (gradient) and communication steps required to reach an $\epsilon > 0$ accurate solution in expectation ($\|r_k\|_2 \leq \epsilon$). We show, using both theoretical and numerical illustrations, that the flexibility provided by the number of communication steps ($n_c$) and the communication probability ($p$) can lead to a reduction in computation and communication complexity as compared to the baseline gradient tracking algorithms for which $n_c=1$ and $p=1$.  We use these results to also perform a comparison among the methods described in \cref{tab: Algorithm_Def}. 

We begin by quantifying the complexity of computation (gradient) steps and the expected complexity of communication steps of the \texttt{RGTA} framework in the following corollary.
\bcorollary \label{col. complexity_results}
    Suppose \cref{asum: convex_and_smooth} holds, and $0 < p < 1$ and $n_c \geq 1$ in \cref{alg: Randomised}. Let $\{x_k\}$ and $\{y_k\}$ denote the iterates generated by \cref{alg: Randomised}, and let the constant step size $\alpha$ be sufficiently small such that $\rho(A(n_c, p)) < 1$. Then, the number of computation (gradient) steps and the expected number of communication steps required to reach a solution that satisfies $\|r_k\|_2 \leq \epsilon$, for $\epsilon > 0$, are
    \begin{equation} \label{eq: complexity}
        \frac{1}{1 - \rho(A(n_c, p))}\log\left(\frac{1}{\epsilon}\right) \quad  
        \mbox{and} \quad 
        \frac{p n_c}{1 - \rho(A(n_c, p))}\log\left(\frac{1}{\epsilon}\right),
    \end{equation}
    respectively.
\ecorollary
\bproof
Since the conditions in \cref{th. general_rate_bound} are satisfied, the result for the number of computation steps $(k)$ required to ensure $\|r_k\|_2 \leq \epsilon$ directly follows. Combining this result with the expected number of communication steps per iteration, which is $pn_c$, completes the proof.
\eproof

We investigate the effects of the number of communication steps ($n_c$) and the communication probability ($p$) on the complexity bounds from \cref{col. complexity_results} analytically and with numerical illustrations. In \cref{fig:complexity_exps_n_c,fig:complexity_exps_p}, we numerically illustrate the computation complexity and expected communication complexity bounds provided in \cref{col. complexity_results}, with $\epsilon = e^{-1}$, for the algorithms in \cref{tab: Algorithm_Def} on a synthetic ill-conditioned problem with condition number $\kappa = 10^4$ (with $\mu$ = $10$, $L$ = $10^5$). We consider three different values of $\beta \in \{0.6, 0.8, 0.9\}$ to simulate different levels of connectivity in the network. The spectral radius $\rho(A(n_c, p))$ is calculated by constructing the matrix $A(n_c, p)$ in \eqref{eq: A_matrices}, using the values specified above, and the step size is tuned to achieve the smallest value for $\rho(A(n_c, p))$. %and tuning the step size $\alpha$ to achieve the smallest value for $\rho(A(n_c, p))$.

\begin{figure}
\centering
    \includegraphics[width = \textwidth]{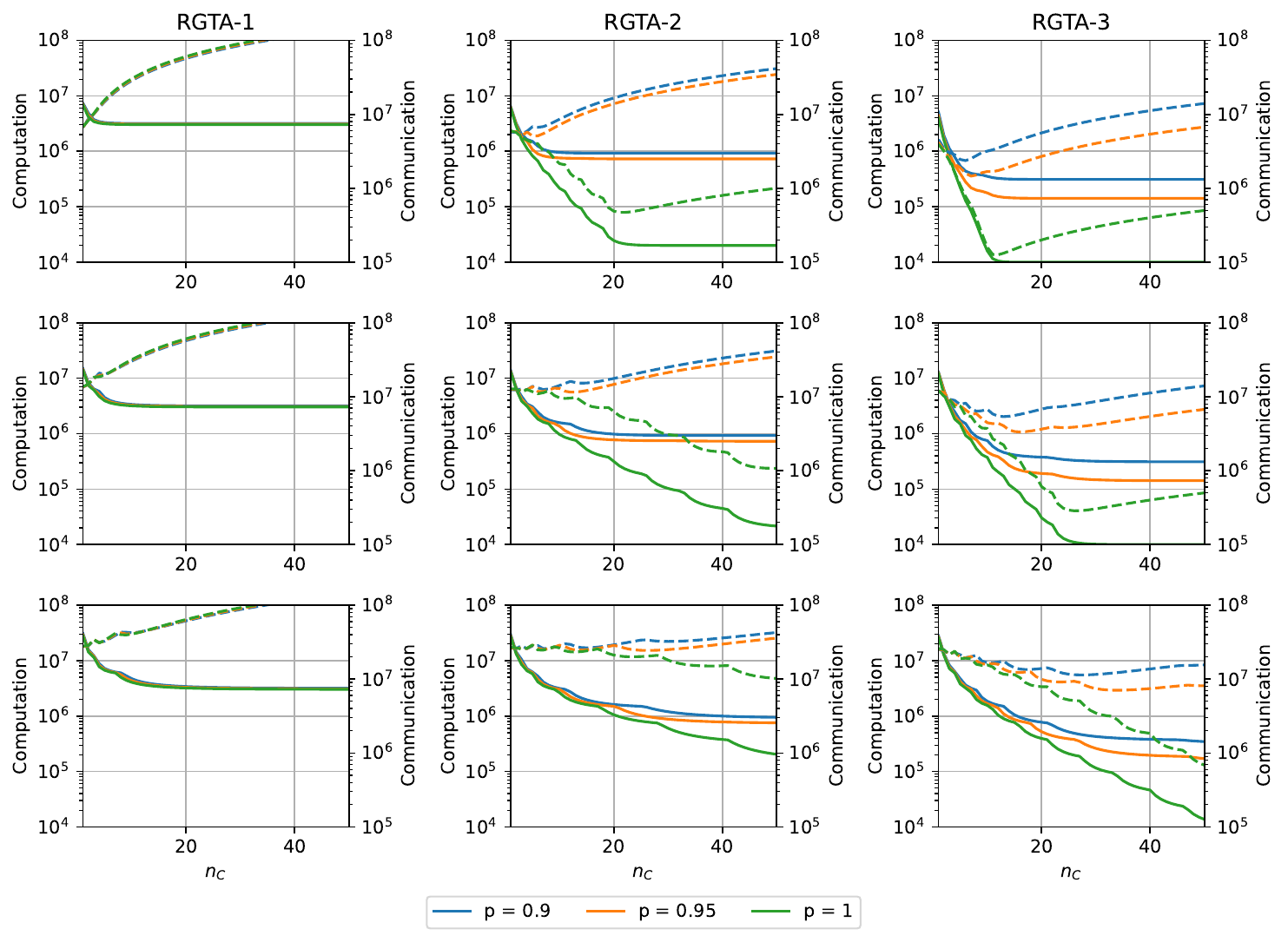}
\caption{Computation complexity (left axis, solid lines) and expected communication complexity (right axis, dashed lines) of \texttt{RGTA-1}, \texttt{RGTA-2} and \texttt{RGTA-3} from \eqref{eq: complexity} for $\epsilon = e^{-1}$ with respect to $n_c$ over an ill conditioned problem ($\mu =10$, $L = 10^5$, $\kappa = 10^4$). Each row of plots represents a value of $\beta$ increasing within the set $\{0.6, 0.8, 0.9\}$ moving down the rows. The computation complexity scale starts from the computation complexity of centralized gradient decent, i.e., $10^4$.}
\label{fig:complexity_exps_n_c}
\end{figure}

\begin{figure}
\centering
    \includegraphics[width = \textwidth]{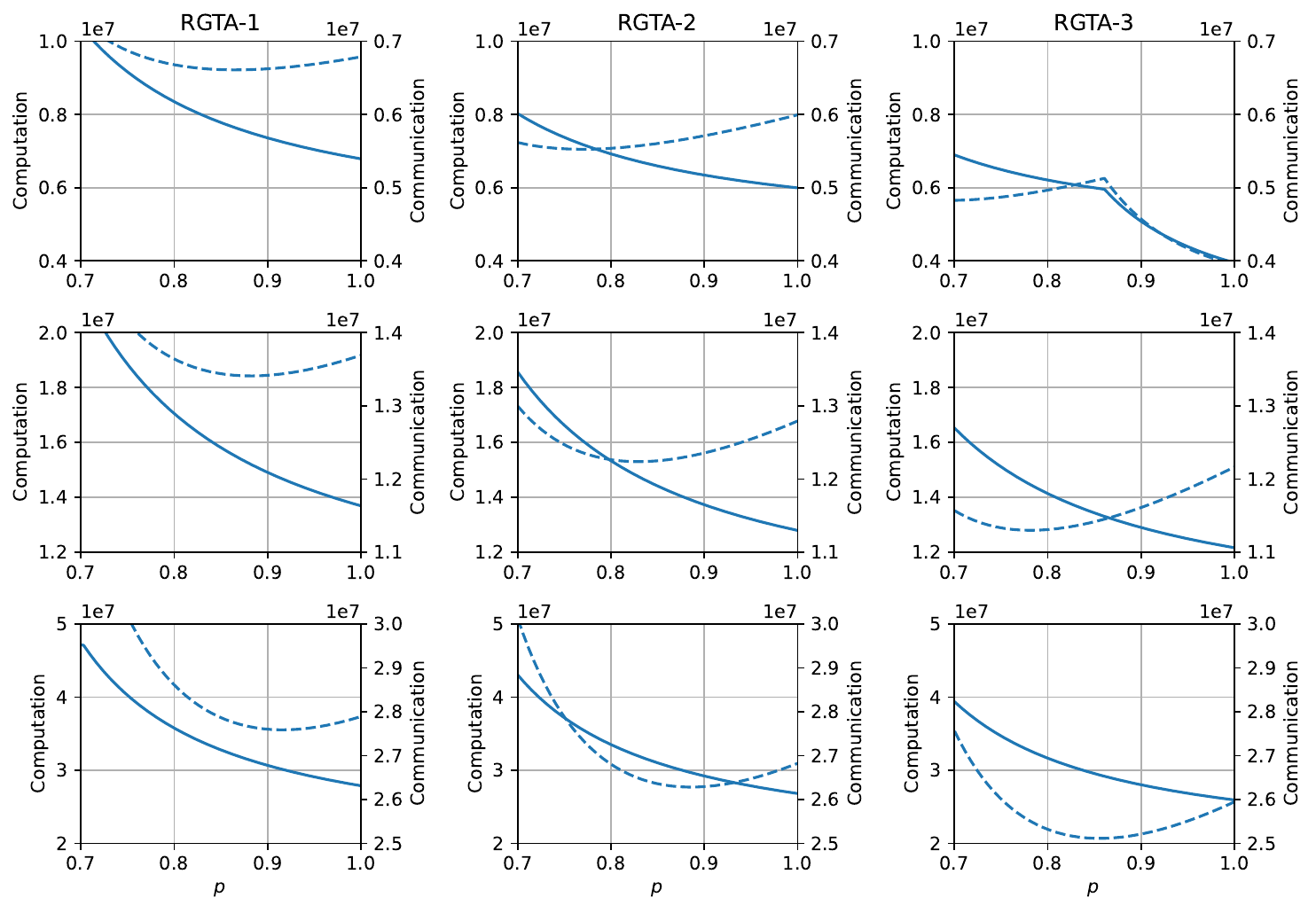}
\caption{Computation complexity (left axis, solid line) and expected communication complexity (right axis, dashed line) of \texttt{RGTA-1}, \texttt{RGTA-2} and \texttt{RGTA-3} from \eqref{eq: complexity} for $\epsilon = e^{-1}$ with respect to $p$ with $n_c = 1$ over an ill conditioned problem ($\mu =10$, $L = 10^5$, $\kappa = 10^4$). Each row of plots represents a value of $\beta$ increasing within the set $\{0.6, 0.8, 0.9\}$ moving down the rows. }
\label{fig:complexity_exps_p}
\end{figure}

\subsubsection*{Computation Complexity} 

We first analyze the effect of the number of communication steps ($n_c$) and communication probability ($p$) on the computation complexity. By \cref{th.incr_rates}, it follows that $\rho(A(n_c, p))$ and the computation complexity decrease as $n_c$ increases; see \eqref{eq: complexity}. The impact of $p$ on the computation complexity can be quantified using the upper bound on the spectral radius in \eqref{eq: general_rate_rate}. The second term in the upper bound decreases as $p$ increases, resulting in a decrease in the computation complexity. 
% This decrease is realized so long as the first term in the upper bound ($1 - \tfrac{\alpha\mu}{2}$) is not dominant. Note that if the first term is the dominant term, then the resulting computation complexity bounds are similar to those of centralized gradient descent. 
Thus, qualitatively, an increase in the number of communication steps, either by increasing $n_c$ or $p$, leads to a reduction in the computation complexity with a natural speed limit of the centralized gradient descent complexity, depicted in the first term of the upper bound \eqref{eq: general_rate_rate}.

The precise quantification of the effect of an increase in the communication volume on the spectral radius and subsequently on the computation complexity bounds depends on the choice of communication matrices, i.e., $\Wmbf_i$ for all $i=1, 2, 3, 4$. As shown in  \cref{th. rate_bound}, the magnitude of reduction differs for \texttt{RGTA-1}, \texttt{RGTA-2} and \texttt{RGTA-3}, particularly for ill-conditioned problems, i.e., $\kappa \gg 1$. For such problems, the upper bounds in \eqref{eq: rates} are dominated by the second terms involving $\kappa$. Since an increase in the number of communication steps, i.e., an increase in $n_c$ or $p$, decreases $\eta$, the effect is predominant in \texttt{RGTA-3} where the entire second term decreases as $\eta$ decreases.  
The effect is weaker in \texttt{RGTA-2} due to the reduced impact of $\eta$ on $\kappa$ (see \eqref{eq: rates}) and weakest in  \texttt{RGTA-1} where $\eta$ does not impact $\kappa$ (see \eqref{eq: rates}). 
This relative ordering of these methods, in terms of the effect of increasing communication volume on the computation complexity bounds, can be inferred by comparing  \texttt{RGTA-1}, \texttt{RGTA-2} and \texttt{RGTA-3} in \cref{fig:complexity_exps_n_c}. As discussed earlier, as $n_c$ increases, the computation complexity initially decreases until it reaches a plateau beyond which further reduction is unattainable. Notably, \texttt{RGTA-1} reaches this plateau with fewer communication steps, i.e., smaller $n_c$ values, as illustrated in the first column of \cref{fig:complexity_exps_n_c}. That said, the computation complexity at this plateau is higher than that of centralized gradient descent (which is $10^4$ and marks the start of the scale for computation complexity in \cref{fig:complexity_exps_n_c}). In contrast, both \texttt{RGTA-2} and \texttt{RGTA-3} achieve the computation complexity of centralized gradient descent, with \texttt{RGTA-3} reaching this level with a relatively smaller number of communication steps ($n_c$) as compared to \texttt{RGTA-2}. 
% Furthermore, increasing $\beta$ (moving down the rows) leads to a higher plateau for the computational complexity value in all three algorithms.
Furthermore, increasing $\beta$ (moving down the rows) requires a higher $n_c$ to achieve the plateau for computational complexity in all three algorithms.
Finally, we also observe that increasing $p$ results in reduced computation complexity for all three algorithms.

\subsubsection*{Expected Communication Complexity}  

Next, we discuss the effect of the number of communication steps ($n_c$) %and $p$ 
on the expected communication complexity. By \cref{col. complexity_results}, both the numerator ($pn_c$) and the denominator ($1 - \rho(A(n_c, p)$)) of the expected communication complexity increase with an increase in $n_c$, since $\rho(A(n_c, p))$ decreases with an increase in $n_c$. However, as $n_c$ increases, the rate at which the numerator increases dominates the rate at which the denominator increases, as the later is bounded above by $1$. Thus, an increase in $n_c$ eventually %overall
increases the expected communication complexity, and there exists an optimal $n_c^* \geq 1$ that depends on the problem specific parameters ($L, \mu, \kappa$) and the choice of communication matrices, i.e., $\Wmbf_i$ for all $i=1, 2, 3, 4$, that leads to the best (lowest) expected communication complexity. This effect is illustrated in \cref{fig:complexity_exps_n_c} where the expected communication complexity either only increases with an increase in $n_c$ (\texttt{RGTA-1}) or decreases initially until 
$\rho(A(n_c, p))$ in the denominator cannot further be decreased 
% the denominator ($1 - \rho(A(n_c, p))$) cannot further be improved 
and then increases as the numerator ($pn_c$) is monotonically increasing with $n_c$ (\texttt{RGTA-2} and \texttt{RGTA-3}).
%This effect is illustrated in \cref{fig:complexity_exps_n_c} where the expected communication complexity \sg{either only increases with an increase in $n_c$ (\texttt{RGTA-1}) or decreases until the computation complexity \rb{($\frac{1}{1 - \rho(A(n_c, p))}$)} cannot be improved and then increases as more communications do not \rb{lead to }%bring 
%improvements after computation complexity plateaus (\texttt{RGTA-2} and \texttt{RGTA-3})}. 
This shows the existence of an optimal $n_c^* \geq 1$ that minimizes expected communication complexity. Moreover, as $\beta$ increases, 
%is increased, 
i.e., the network is less connected, the optimal $n_c^*$ value increases as more communications are required to reach the plateau for computation complexity. This also results in an increase in the minimum expected communication complexity obtained at $n_c^*$ as the numerator ($pn_c$) increases. %in \eqref{eq: complexity} increases.

Finally, we analyze the effect of the communication probability ($p$) on the expected communication complexity. As $p$ decreases, both the numerator ($pn_c$) and the denominator ($1 - \rho(A(n_c, p)$)) of the expected communication complexity (given in \cref{col. complexity_results}) decrease. That said, the denominator decreases at a slower rate than the numerator, resulting in a decrease in the expected communication complexity as $p$ decreases from one to zero. We should note that the framework fails to converge when $p = 0$, as in this regime the nodes never communicate and consensus cannot be achieved. Thus, an optimal $0<p^*<1$ exists that minimizes the expected communication complexity, where the specific value depends on problem specific parameters ($L, \mu, \kappa$) and the choice of the communication matrices, i.e., $\Wmbf_i$ for all $i=1, 2, 3, 4$. Due to the scale at which expected communication complexity changes with changes in $p$, these effects are not too visible in \cref{fig:complexity_exps_n_c} and we point the reader to \cref{fig:complexity_exps_p} that illustrates these changes clearly. In addition to the previously discussed effect of reduced communication complexity with decreases in $p$, and the relative ordering of \texttt{RGTA-1}, \texttt{RGTA-2} and \texttt{RGTA-3}, we observe the existence of an optimal $0 < p^* < 1$ that minimizes the expected communication complexity. Furthermore, we observe that the optimal $p^*$ and the optimal expected communication complexity achieved by \texttt{RGTA-3} are lower than those of \texttt{RGTA-2}, which, in turn, are lower than those of \texttt{RGTA-1}. Additionally, the optimal $p^*$ increases as network connectivity decreases (or $\beta$ increases) moving down the rows\footnote{An exception within these results is observed in the case of \texttt{RGTA-3} for $\beta = 0.6$, where the expected communication complexity increased with a decrease in $p$ from one. We believe this is an artifact of employing overly pessimistic and loose bounds in the matrix $A(n_c, p)$ from \eqref{eq: general_A} when the network is well connected. We provide numerical examples of reduction in expected communication complexity in \cref{sec.num_exp} on networks with high connectivity (low $\beta$).}.

\bigskip

In summary, as the number of communication steps  and/or the communication probability increases the computation complexity decreases. On the other hand, as the number of communication steps  and/or the communication probability increases the expected communication complexity (eventually) increases. The results presented in this section demonstrate that the  parameters in the \texttt{RGTA} framework ($n_c$ and $p$) indeed provide the necessary flexibility required to balance the computation complexity and the expected communication complexity. The existence of optimal choices $n_c^*$ and $p^*$ reveals that the algorithmic framework is capable of operating in regimes with reduced computation and/or expected communication complexity (or a combination) depending on the problem  and application.

\section{Numerical Experiments}\label{sec.num_exp}

In this section, we illustrate the empirical performance of the methods defined in \cref{tab: Algorithm_Def} using Python implementations\footnote{Our code will be made publicly available upon publication of the manuscript. Github repository: \texttt{\url{https://github.com/Shagun-G/Gradient-Tracking-Algorithmic-Framework}}. 
}. 
We show, via multiple example problems, the benefits of balancing % the number of communication and computation steps 
the number of communication and computation steps
and establish a relative ordering among the methods defined in \cref{tab: Algorithm_Def}. 
We also compare the performance of our proposed methods with that of popular algorithms, e.g., FedAvg \cite{mcmahan2017communication}, Scaffold \cite{karimireddy2020scaffold} and Scaffnew \cite{mishchenko2022proxskip}.

We considered two problem classes: %to illustrate these goals: 
$(1)$ synthetic strongly convex quadratic problems (\cref{sec : quads}); and, $(2)$ binary classification logistic regression problems over the a9a and w8a datasets \cite{CC01a} (\cref{sec : logistic}). We solved these problems over three network structures (different mixing matrices $\Wmbf$) with $n = 16$ nodes: $(1)$ a fully connected network where all pairs of nodes are connected with equal weights (same as a %server client-setup 
client-server setup \cite{mcmahan2017communication}, $\Wmbf = \frac{1_n1_n^T}{n}$, $\beta = 0$); $(2)$ a connected star network where all nodes are connected to a single central node ($\beta = 0.95$); and, $(3)$ a connected line network where all nodes have two %2 
neighbors except for the two endpoints which have a single neighbor ($\beta = 0.998$). 
The star and line networks have low connectivity (i.e., high $\beta$ value). 
We should note that the performance of \cref{alg: Randomised} with multiple communication steps is equivalent to that of the same algorithm over a network with higher connectivity (i.e., lower $\beta$). 

The performance of the methods was measured in terms of the optimization error ($\|\bar{x}_k - x^*\|_2$) and the consensus error ($\|\xmbf_k - \xbb_k\|_2$) with respect to the number of gradient computations (``Gradients'') at each node and the number of communication rounds (``Communications''), where one communication round involves all nodes sharing information with their neighbors. The optimal solution $x^*$ was obtained analytically for the quadratic problems and by running gradient descent in the centralized setting to high accuracy, i.e., $\|\nabla f(x^*)\|_2 \leq 10^{-12}$, for the logistic regression problems.
We do not report the consensus error in the auxiliary variable $\ymbf_k$ ($\|\ymbf_k - \ybb_k\|_2$) as this measure does not provide any significant additional insights about the performance of the algorithms. We also do not report consensus error ($\|\xmbf_k - \xbb_k\|_2$) for problems over the fully connected network since it is always zero after every round of communication.
%as it can always be brought down to zero with one communication round.

We implemented and compared %against 
our approach to FedAvg, Scaffold and Scaffnew. For FedAvg and Scaffold we only present results for fully-connected networks as these methods were designed to operate in this setting. These methods were also given access to all the nodes in each round of communication instead of sampling a subset of the nodes and the gradient of the complete dataset at each node was computed instead of a stochastic estimate. 
This implementation ensured a fair comparison with the \texttt{RGTA} framework.

The methods defined in \cref{tab: Algorithm_Def} were denoted by \texttt{RGTA}$-i(n_c,p)$, $i = 1, 2, 3$, where $n_c$ is the number of communications and $p$ is the probability of communication. The \texttt{RGTA} parameters were tested over the following sets, $n_c \in \{1, 5, 10, 20, 50\}$ and $p \in \{0.01, 0.1, 0.2, 0.4, 0.6, 0.8, 1\}$. 
As a result of these parameter sets, we recover popular gradient tracking methods as special cases that provide a baseline for the performance of the gradient tracking methodology, e.g., DIGing (\texttt{RGTA}$-1(1, 1)$)~\cite{nedic2017geometrically}, NEXT (\texttt{RGTA}$-2(1, 1)$) \cite{di2016next},  and Aug-DGM (\texttt{RGTA}$-3(1, 1)$) \cite{xu2015augmented}. 
% As a result of these parameter sets, we recover as special cases and compared with popular gradient tracking methods that provide a baseline for the performance of general flexible gradient tracking methodology, e.g.,  DIGing (\texttt{RGTA}$-1(1, 1)$)~\cite{nedic2017geometrically}, NEXT (\texttt{RGTA}$-2(1, 1)$) \cite{di2016next},  and Aug-DGM (\texttt{RGTA}$-3(1, 1)$) \cite{xu2015augmented} . 
The probability of communication for Scaffnew was tested over the same set. % $\{0.01, 0.1, 0.2, 0.4, 0.6, 0.8, 1\}$. 
The number of local steps at each node in FedAvg and Scaffold was tested over the set $\{1, 5, 10, 20, 50\}$. The step sizes (one for \texttt{RGTA}, one for FedAvg, two for Scaffold and two for Scaffnew) were tuned over the set $\{2^{-t} | t = 0, 1, ..., 20\}$. All algorithms, for all problems, at all nodes were initialized with the zero vector (i.e., $\xmbf_0 = \mathbf{0}$). We should note that initializing all nodes at the same value is not required by \texttt{RGTA} but is a requirement for Scaffnew.

%%%%%%%%%%%%%%%%%
%%%%%%%%%%%%%%%%%
\subsection{Quadratic Problems}\label{sec : quads}

%In this section we consider
We first consider quadratic problems of the form
\begin{align*} %\label{eq : quad problem}
    f(x) = \frac{1}{n} \sum_{i=1}^n \frac{1}{2}x^TQ_ix + v_i^Tx,
\end{align*}
generated using the procedure described in \cite{mokhtari2016network} where $Q_i \in \mathbb{R}^{10 \times 10}$, $Q_i \succ 0$ and $v_i \in \mathbb{R}^{10}$ is the local information at each node $i \in \{1, 2, .., n\}$ with $n=16$. Each local problem is strongly convex and $\kappa$ denotes the global condition number.
% Each local problem is strongly convex with the global condition number $\kappa \approx 10^4$. 

\begin{figure}
\centering
\begin{subfigure}{\textwidth}
    \includegraphics[width=\textwidth]{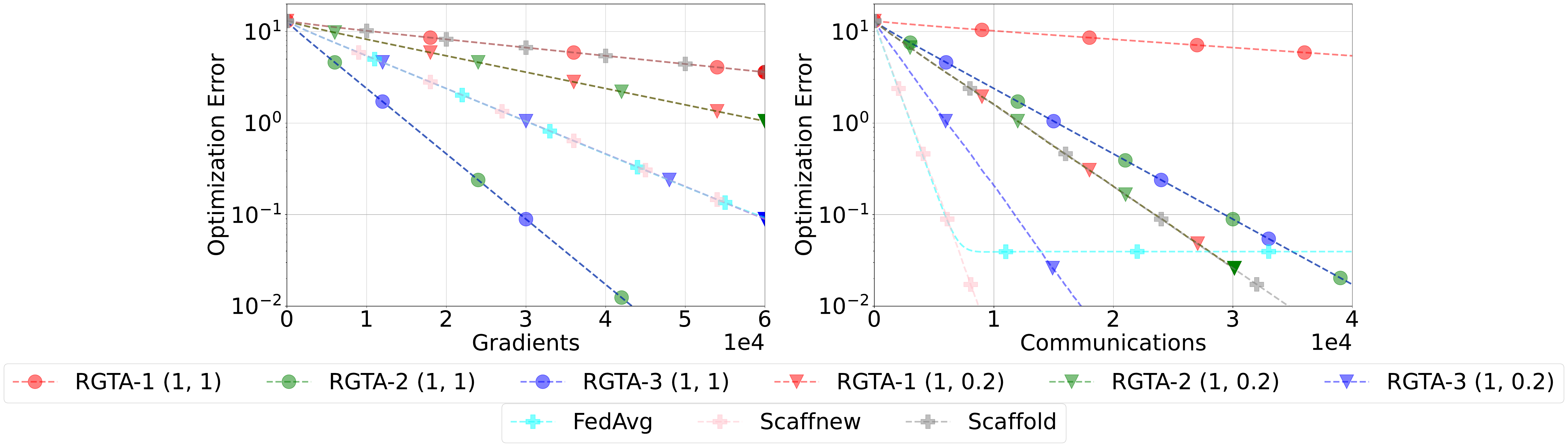} 
    \caption{Fully Connected Network, $\beta = 0$}
    \label{fig:Quad_full}
\end{subfigure}
\begin{subfigure}{\textwidth}
    \includegraphics[width=\textwidth]{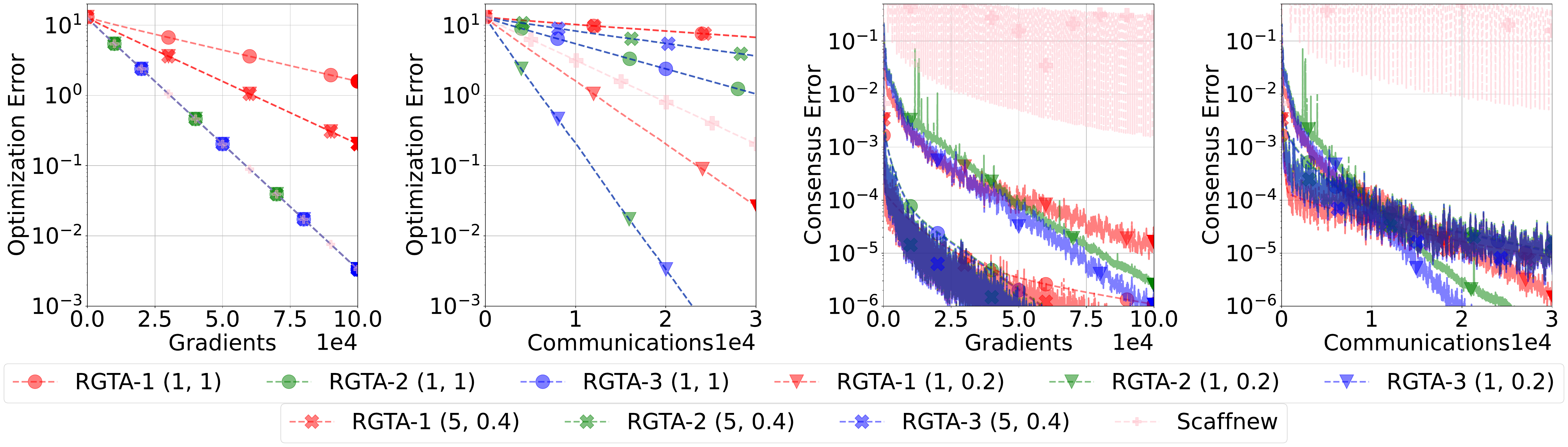} 
    \caption{Start Network, $\beta = 0.95$}
    \label{fig:Quad_star}    
\end{subfigure}
\begin{subfigure}{\textwidth}
    \includegraphics[width=\textwidth]{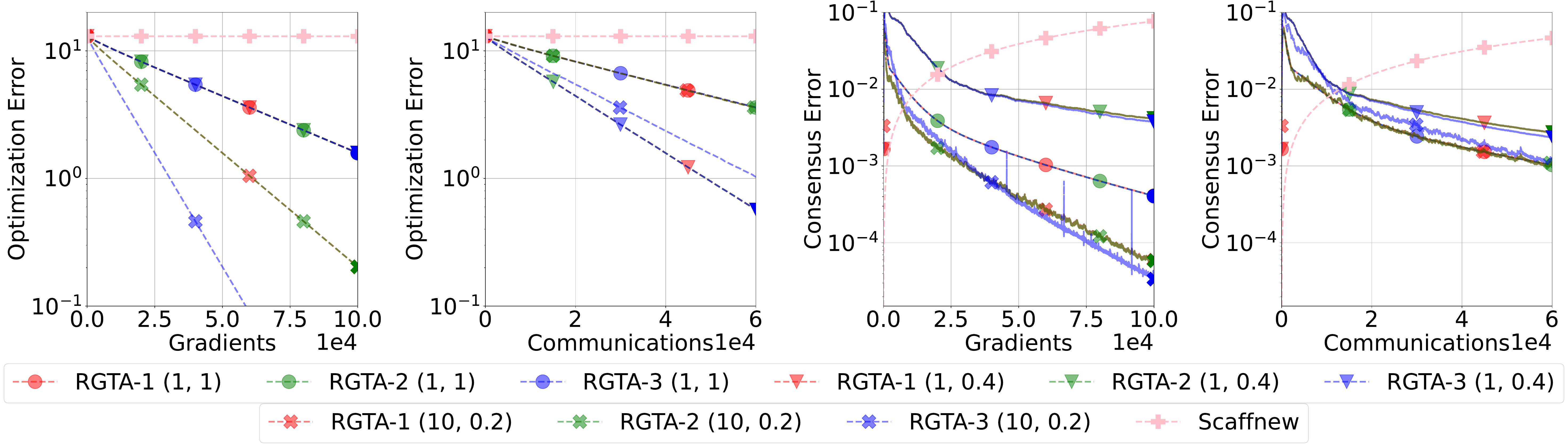} 
    \caption{Line Network, $\beta = 0.999$}
    \label{fig:Quad_line}    
\end{subfigure}
\caption{Optimization Error ($\|\xbar_k - x^*\|_2$) and Consensus Error ($\|\xmbf_k - \xbb_k\|_2$) of \texttt{RGTA-1}, \texttt{RGTA-2} and \texttt{RGTA-3} with respect to number of gradient evaluations and communication rounds for a synthetic quadratic problem ($n = 16$, $d = 10$, $\kappa \approx 10^4$) over: (a) fully connected network ($\beta =  0$); (b) star network ($\beta = 0.95$); and, (c) line network ($\beta = 0.999$).}
\label{fig:Quad_exps}
\end{figure}

%In Figs. \ref{fig:Quad_full}, \ref{fig:Quad_star} and  \ref{fig:Quad_line}, we 
% \rb{Figs. \ref{fig:Quad_full}, \ref{fig:Quad_star} and  \ref{fig:Quad_line}} 
\cref{fig:Quad_full,fig:Quad_star,fig:Quad_line} %\cref{fig:Quad_exps} 
show the performance of the algorithms from \cref{tab: Algorithm_Def} on fully connected, star and line networks, respectively. 
Several key observations can be made. First, when comparing the base methods (\texttt{RGTA-i}(1, 1), $i=1,2,3$), we observe that \texttt{RGTA-3}(1, 1) either outperforms or matches the performance of \texttt{RGTA-2}(1, 1), which in turn either outperforms or matches the performance of \texttt{RGTA-1}(1, 1). Second, the differences in the performance is reduced when the network is sparsely connected (higher $\beta$). In fact, all three methods exhibit similar performance on the line graph, as seen in \cref{fig:Quad_line}. This observation is consistent with the theory presented in \cref{sec.theory}. Third, the performance of all three methods in terms of the communication rounds improves when the probability of communication ($p$) is reduced from one. %, as illustrated in \cref{fig:Quad_exps}. 
This improvement is reflected via the reduction in the optimization error with minimal impact on the consensus error. % while not impacting consensus error significantly. 
Fourth, as the number of communication steps ($n_c$) is increased, for any  probability of communication ($p$) that is less than one, the performance of all three methods improves in terms of the number of gradient evaluations. This improvement %time the reduction 
is observed in the consensus error with minimal impact on the optimization error. %while not having a significant impact on the optimization error, as seen in Figs. \ref{fig:Quad_star} and \ref{fig:Quad_line}. 
The empirical results of the proposed framework are competitive with popular methods such as FedAvg, Scaffold and Scaffnew over the fully connected network; see  \cref{fig:Quad_full}.
% The performance of the proposed framework is competitive with popular methods such as FedAvg, Scaffold and Scaffnew over the fully connected network; see  \cref{fig:Quad_full}. % Furthermore, in  \cref{fig:Quad_star,fig:Quad_line}
% we compared our framework with Scaffnew over star and line networks, respectively. We observe that Scaffnew was non-convergent on the line graph, as seen in \cref{fig:Quad_line}. We suspect this issue is related to hyperparameter tuning, as, theoretically, the method should be convergent for this problem. 
% Nevertheless, we have tested a sufficiently large set of hyper-parameters, and presented the results that yielded the best performance.
We acknowledge that in \cref{fig:Quad_line}, Scaffnew is non-convergent on the line graph and attribute this to hyper-parameter tuning. That said, we tested a sufficiently large and wide set of hyper-parameters to warrant presenting these results. Overall, the performance of the proposed algorithms is robust and efficient across a wide variety of parameter values and network structures.

\begin{figure}
\centering
\includegraphics[width=\textwidth]{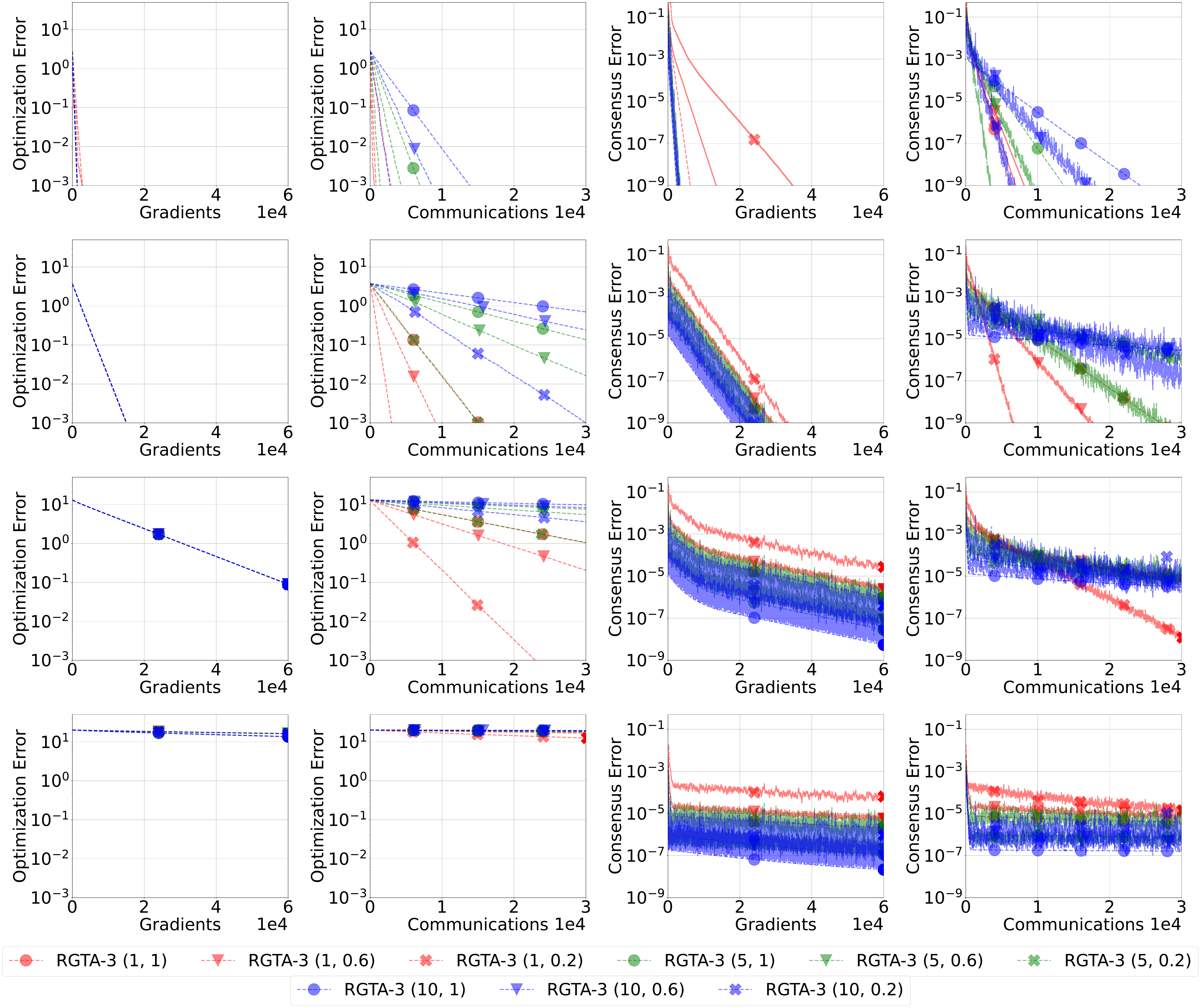} 
\caption{Optimization Error ($\|\xbar_k - x^*\|_2$) and Consensus Error ($\|\xmbf_k - \xbb_k\|_2$) of \texttt{RGTA-3} with respect to number of gradient evaluations and communication rounds over star network ($\beta =  0.95$) for four synthetic quadratic problems, each with $n = 16$ and $d = 10$ but different $\kappa$. Each row represents results for one quadratic problem and the condition number increases from top to bottom as  $\kappa \approx 10^2, 10^3, 10^4, 10^5$.}
\label{fig:Quad_sensitivity}
\end{figure}

In \cref{fig:Quad_sensitivity}, we present the sensitivity %analysis 
of \texttt{RGTA-3} over the star network with respect to the condition number of the quadratic problem ($\kappa$). We % choose 
chose \texttt{RGTA-3} because it exhibits the best performance among the three methods. Each row in the figure corresponds to one problem, with the condition number increasing across the set $\kappa \in \{10^2, 10^3, 10^4, 10^5\}$ from top to bottom.  In terms of the optimization error, when $\kappa \approx 10^2$ (well-conditioned), the problem is easy enough to not require the flexibility provided by \texttt{RGTA-3} via the parameters $n_c$ and $p$. %warrant the use of $n_c$ and $p$ to exploit the flexibility provided by \texttt{RGTA-3}. 
However, when $\kappa \approx 10^3$ or $10^4$, we observe the effects discussed earlier in this subsection related to \cref{fig:Quad_exps}. % discussed in the previous 
%previously discussed in thissubsection with respect to $n_c$ and $p$. 
These effects are not observed when $\kappa \approx 10^5$ (ill-conditioned), where the problem is too difficult to gain any improvements from the flexibility.  
This observation shows the limitations of the flexibility explained by \cref{th. general_rate_bound}, and the myopic nature of first-order methods. That is, changing $p$ and $n_c$ only affects the coefficient of $\kappa$ in the convergence rate, and so if $\kappa$ is large enough, these changes in the coefficient are less impactful. 
% This observation is explained by \cref{th. general_rate_bound}, where changing $p$ and $n_c$ changes the coefficients of $\kappa$ in \eqref{eq: general_rate_rate}. 
Thus, in general, the performance of the \texttt{RGTA} framework on ill-conditioned problems can benefit from the usage of a more densely connected network (lower $\beta$).  
On the other hand, the consensus error improves as $\kappa$ increases as the step sizes employed and local steps taken at all nodes become smaller with an increase in the condition number. This results in minimal change in the local decision variables and low consensus error as all nodes are initialized with same (zero) vector.

\begin{figure}[H]
\centering
\includegraphics[width = \textwidth]{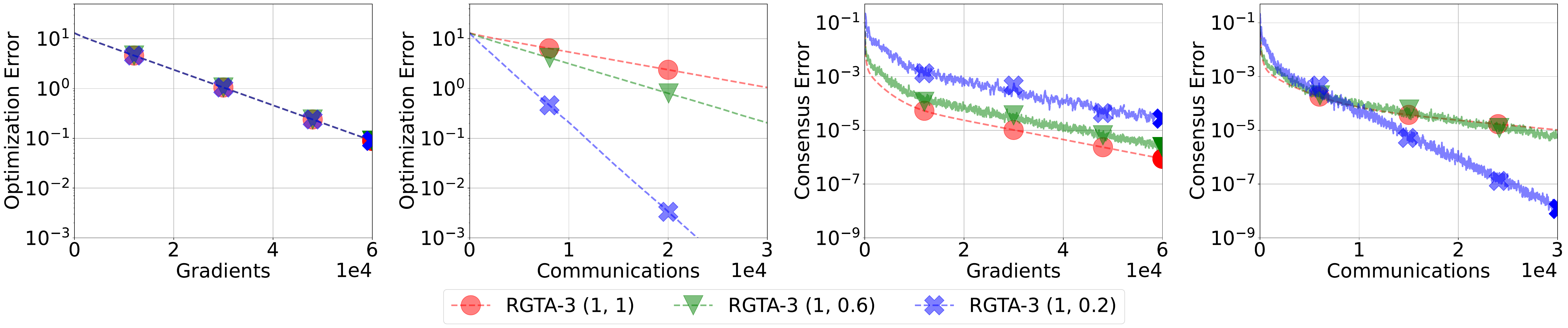}
\includegraphics[width = \textwidth]{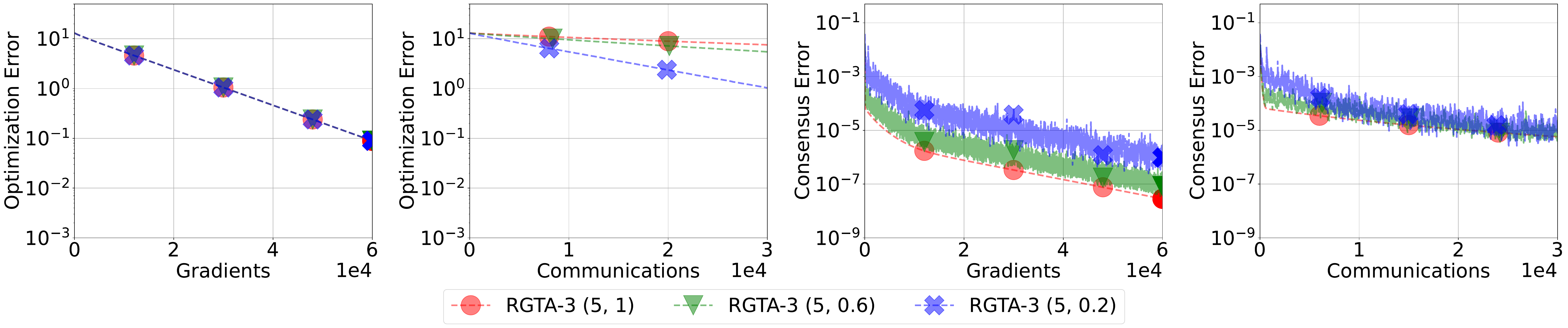}
\includegraphics[width = \textwidth]{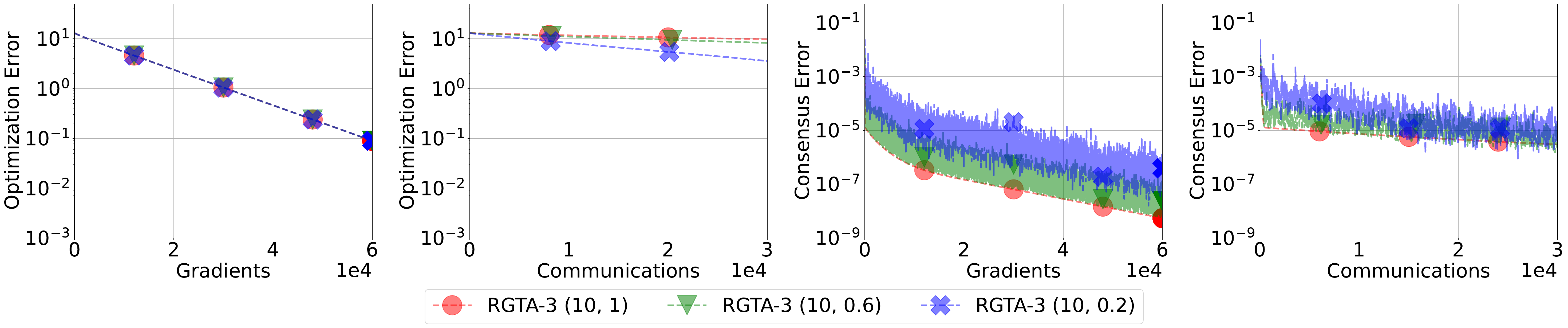}
\caption{Optimization Error ($\|\xbar_k - x^*\|_2$) and Consensus Error ($\|\xmbf_k - \xbb_k\|_2$) of \texttt{RGTA-3} with respect to number of gradient evaluations and communication rounds for a synthetic quadratic problem ($n = 16$, $d = 10$, $\kappa \approx 10^4$) over star network ($\beta =  0.95$).}
\label{fig:Quad_star_sentivity_detailed}
\end{figure}

In \cref{fig:Quad_star_sentivity_detailed}, we further illustrate the sensitivity of \texttt{RGTA-3} with respect to $n_c$ and $p$ on one of the quadratic problems discussed in \cref{fig:Quad_sensitivity}, where $\kappa \approx 10^4$. Each row within this figure shows the performance for different $n_c$ values, and within each row the figure demonstrates the impact of changing $p$ while keeping $n_c$ constant. 
%Each row within this figure demonstrates the impact of changing $p$ \sg{while keeping $n_c$ constant.} 
%\sg{As we move down the rows, it illustrates the effect of increasing $n_c$. } 
In the first column, we observe that changing $p$ or $n_c$ had negligible effects on the optimization error with respect to the number of gradient evaluations. Similarly, there is negligible effect on the consensus error with respect to number of communications (fourth column). However, in the second column, we observe that the optimization error with respect to number of communications deteriorates with increased communication, i.e., it deteriorates by increasing $n_c$ and improves by decreasing $p$. On the other hand, the consensus error with respect to gradient evaluations improves with increased number of communications, highlighting the contrasting effect of the communication parameters on different performance metrics as illustrated by the discussion in \cref{sec.compl}.
% \rb{In the first column, we observe that changing $p$ or $n_c$ had negligible effects on the optimization error with respect to the number of gradient evaluations. However, in the second and fourth column, we observe that both the optimization error and the consensus error deteriorate with increased communication, i.e.,  either by increasing $p$ within a plot or by increasing $n_c$ as we move down the column.  On the other hand, the consensus error with respect to gradient evaluations improved with increased communication, highlighting the contrasting effect of communication parameters on different performance metrics as illustrated by the discussion in \cref{sec.compl}.}

% Shagun's second version....
% In column one, we observe that changes in $p$ or $n_c$ do not influence the optimization error with respect to number of gradients.
% In column two, the optimization error with respect to number of communications improves with a decrease in $p$ within plots (same $n_c$). 
% At the same time, moving down the column (increasing $n_c$), deteriorates the performance.
% Column three shows reduced consensus error with respect to number of gradients with increased communication, i.e., increase in $p$ (within a plot) or increase in $n_c$ (moving down the column). Thus, column two and three show the contrasting effects of $p$ and $n_c$, highlighting the trade-off between the two parameters.}

\subsection{Binary Classification Logistic Regression}\label{sec : logistic}
%In this section 
We consider $\ell_2$-regularized binary classification logistic regression problems of the form
\begin{align*} %\label{eq : logistic problem}
    f(x) &= \frac{1}{n} \sum_{i=1}^n \frac{1}{n_i}\log(1 + e^{-b_iA_ix}) + \frac{1}{n_i}\|x\|_2^2, 
\end{align*}
where each node $i \in \{1, 2, .., n\}$ has a portion of data samples $A_i \in \mathbb{R}^{n_i \times d}$ and corresponding labels $b_i \in \{-1, 1\}^{n_i}$. Experiments were performed over the a9a dataset ($n = 16$, $d = 123$, $\sum_{i=1}^n n_i = 32,561$) and the w8a dataset ($n = 16$, $d = 300$, $\sum_{i=1}^n n_i = 49,749 $) \cite{CC01a}. 

\cref{fig:a9a_exps,fig:w8a_exps} show the performance of the algorithms from \cref{tab: Algorithm_Def} on logistic regression problems on the a9a and w8a datasets, respectively, over fully connected, star and line networks. We find that similar observations to those made in \cref{sec : quads} also hold for this set of experiments. Namely, $(1)$ the relative ordering among algorithms (\texttt{RGTA-3} at least as good as \texttt{RGTA-2}, and \texttt{RGTA-2} at least as good as \texttt{RGTA-1}); $(2)$ the reduction of communication probability ($p$) improves performance with respect to number of communication rounds depicted in the optimization error; $(3)$ an increase in number of communication steps ($n_c$) results in an improvement, mainly seen in the consensus error, with respect to number of gradient evaluations. These effects are %more 
amplified for networks with higher connectivity %as the network is more connected 
(lower $\beta$).

%Figures \ref{fig:w8a_full}, \ref{fig:w8a_star} and \ref{fig:w8a_line} show the performance of \texttt{RGTA-1}, \texttt{RGTA-2} and \texttt{RGTA-3} for logistic regression on w8a dataset over a fully connected, star and line network respectively.
% Figs. \ref{fig:a9a_full}, \ref{fig:a9a_star} and \ref{fig:a9a_line} (figs. \ref{fig:w8a_full}, \ref{fig:w8a_star} and \ref{fig:w8a_line}) show the performance of algorithms from \cref{tab: Algorithm_Def} on a logistic regression problem on the a9a (w8a) dataset over a fully connected, star and line network,  respectively.  
% We find similar observations to \cref*{sec : quads} here. The relative ordering among algorithms (\texttt{RGTA-3} at least as good as \texttt{RGTA-2}, at least as good as \texttt{RGTA-1}) \rb{is also evident here.} %can be observed here as well. 
% The reduction of communication probability $p$ improves performance with respect to number of communication rounds depicted in the optimization error. The increase in number of communication steps $n_c$ shows improvement with respect to number of gradient evaluations mainly seen in the consensus error. These affects are %more 
% amplified as the network is more connected (lower $\beta$).

Furthermore, \cref{fig:a9a_full,fig:w8a_full}, show that on logistic regression problems on the a9a and w8a datasets over fully connected networks, the performance of \texttt{RGTA-3}(1, $p$) is competitive to popular algorithms from the literature. 
%\rb{From Figs. \ref{fig:a9a_full} and \ref{fig:w8a_full}, \rb{we observe that} \texttt{RGTA-3}(1, $p$) is competitive to the algorithms in literature over fully connected networks for logistic regression problem on a9a and w8a datasets respectively.} 
%From figs. \ref{fig:a9a_full} and \ref{fig:w8a_full}, \rb{we observe that} \texttt{RGTA-3}(1, $p$) provides competitive performance to the algorithms in literature over fully connected networks for logsitic regression problem over a9a and w8a datasets respectively. 
The empirical advantages (and improvement in performance) of the \texttt{RGTA} framework as compared to the baseline performance of popular gradient tracking methods is highlighted in Figs. \ref{fig:a9a_star}, \ref{fig:a9a_line}, \ref{fig:w8a_star} and \ref{fig:w8a_line}. 
%In Figs. \ref{fig:a9a_star}, \ref{fig:a9a_line}, \ref{fig:w8a_star} and \ref{fig:w8a_line}, the \texttt{RGTA} framework demonstrates improvement over the baseline performance of popular gradient tracking methods. 
That said, we note that Scaffnew outperforms the \texttt{RGTA} framework on the line network experiments. %However, we observe that Scaffnew performs extremely well in comparison. 
While we do not fully understand the reason for this performance (and disparity as compared to the results on the quardratic problems), considering the tuning effort required for Scaffnew (including two step sizes for the decentralized setting), we believe that the proposed \texttt{RGTA} framework is a robust and competitive alternative. 
%offers advantages in comparison. }
%\sg{While we do not understand this disparity in the performance of Scaffnew between quadratic and logistic regression problems, considering the tuning effort required for Scaffnew (including two step sizes for the decentralized setting), we believe that the proposed \texttt{RGTA} framework offers advantages in comparison.}
% In Figs. \ref{fig:a9a_star}, \ref{fig:a9a_line}, \ref{fig:w8a_star} and \ref{fig:w8a_line}, we also compare \rb{the performance of} our framework with \rb{that of} Scaffnew \rb{in} %over
% decentralized settings. We \rb{observe} %see 
% Scaffnew performs extremely well for logistic regression. We do not yet understand the reason for this disparity in the performance of Scaffnew between quadratic and logistic regression problems. \rb{However, considering the tuning effort required for Scaffnew (including two step sizes for the decentralized setting), we believe that the proposed \texttt{RGTA} framework offers advantages in comparison.}
%But given the tuning effort required for Scaffnew (two step sizes for decentralized setting), we believe the proposed \texttt{RGTA} framework has merit in comparison.

\begin{figure}[H]
\centering
\begin{subfigure}{\textwidth}
    \includegraphics[width=\textwidth]{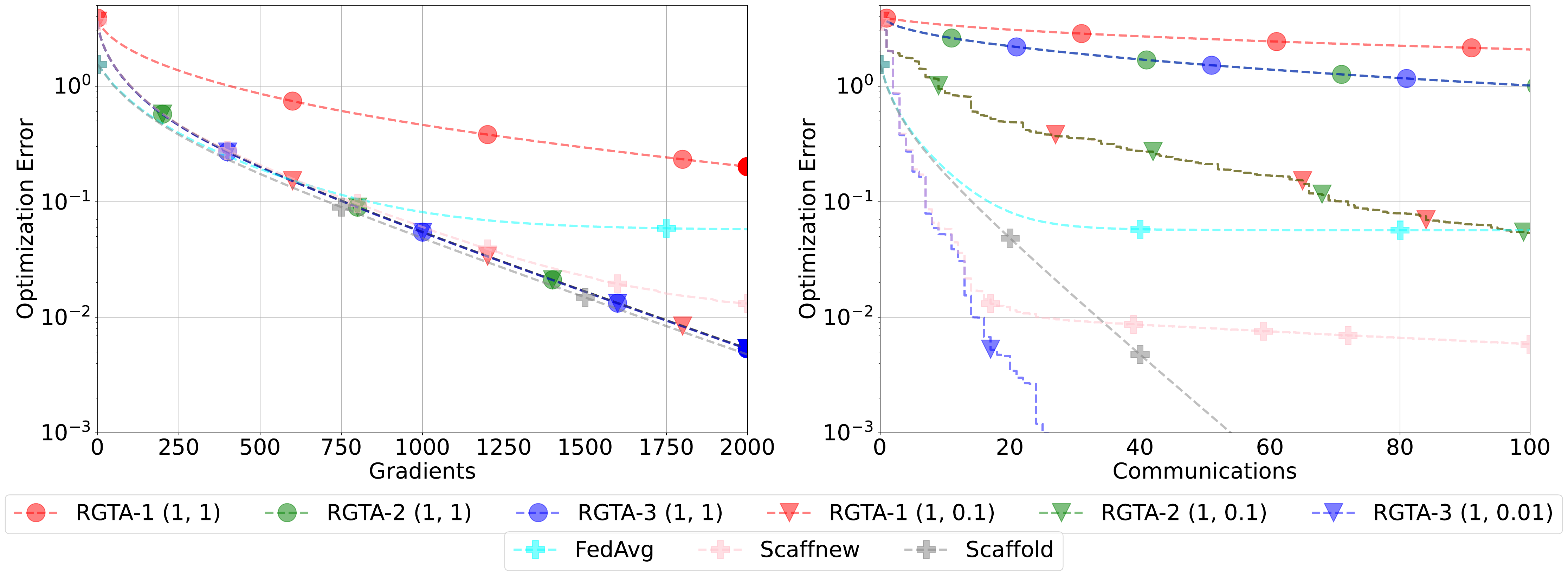}
    \caption{Fully Connected Network, $\beta = 0$}
    \label{fig:a9a_full}
\end{subfigure}

\begin{subfigure}{\textwidth}
    \includegraphics[width=\textwidth]{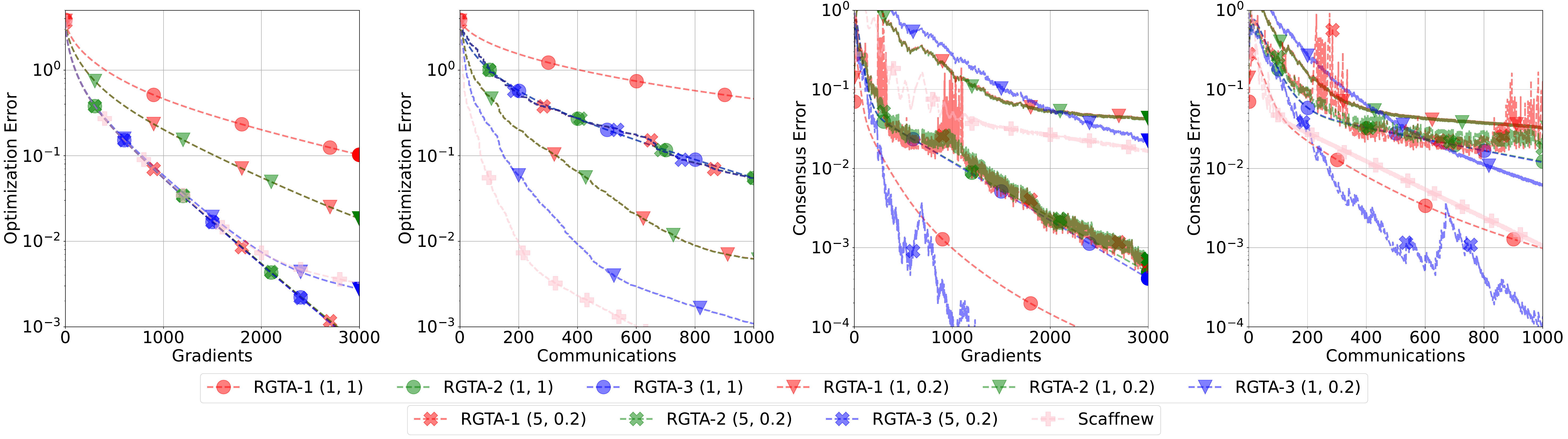} 
    \caption{Start Network, $\beta = 0.95$}
    \label{fig:a9a_star}
\end{subfigure}

\begin{subfigure}{\textwidth}
    \includegraphics[width=\textwidth]{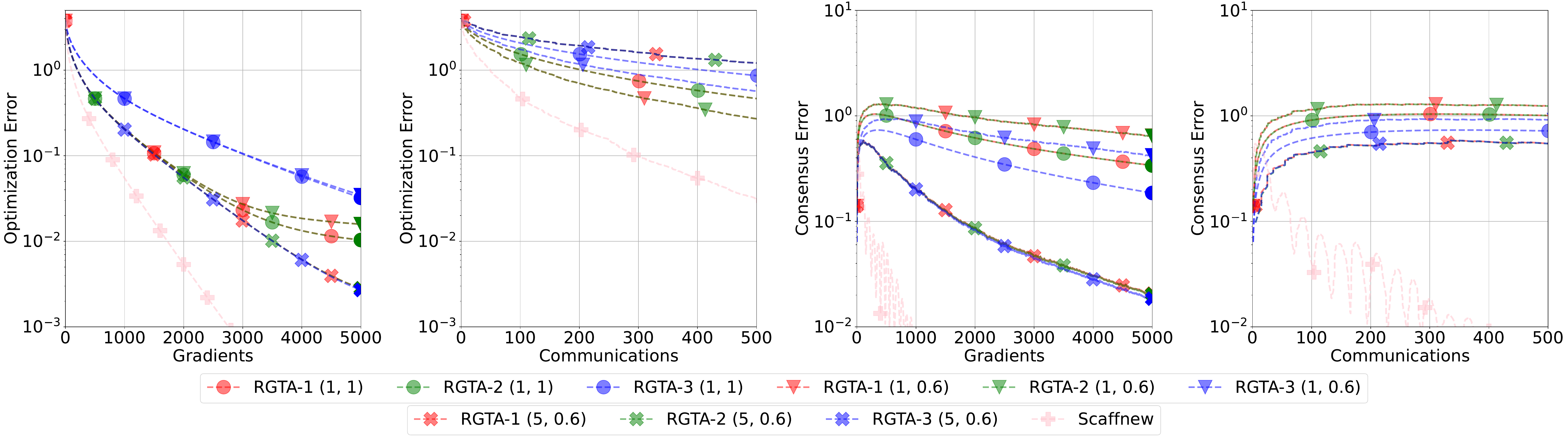}   
    \caption{Line Network, $\beta = 0.999$}
    \label{fig:a9a_line}
\end{subfigure}
\caption{Optimization Error ($\|\xbar_k - x^*\|_2$) and Consensus Error ($\|\xmbf_k - \xbb_k\|_2$) of \texttt{RGTA-1}, \texttt{RGTA-2} and \texttt{RGTA-3} with respect to number of gradient evaluations and communication rounds for binary logistic regression on the a9a dataset ($n = 16$, $d = 123$, $\sum_{i=1}^n n_i = 32,561$) over; (a) fully connected network ($\beta =  0$); (b) star network ($\beta = 0.95$); and, (c) line network ($\beta = 0.999$).}
\label{fig:a9a_exps}
\end{figure}

\begin{figure}[H]
\centering
\begin{subfigure}{\textwidth}
    \includegraphics[width=\textwidth]{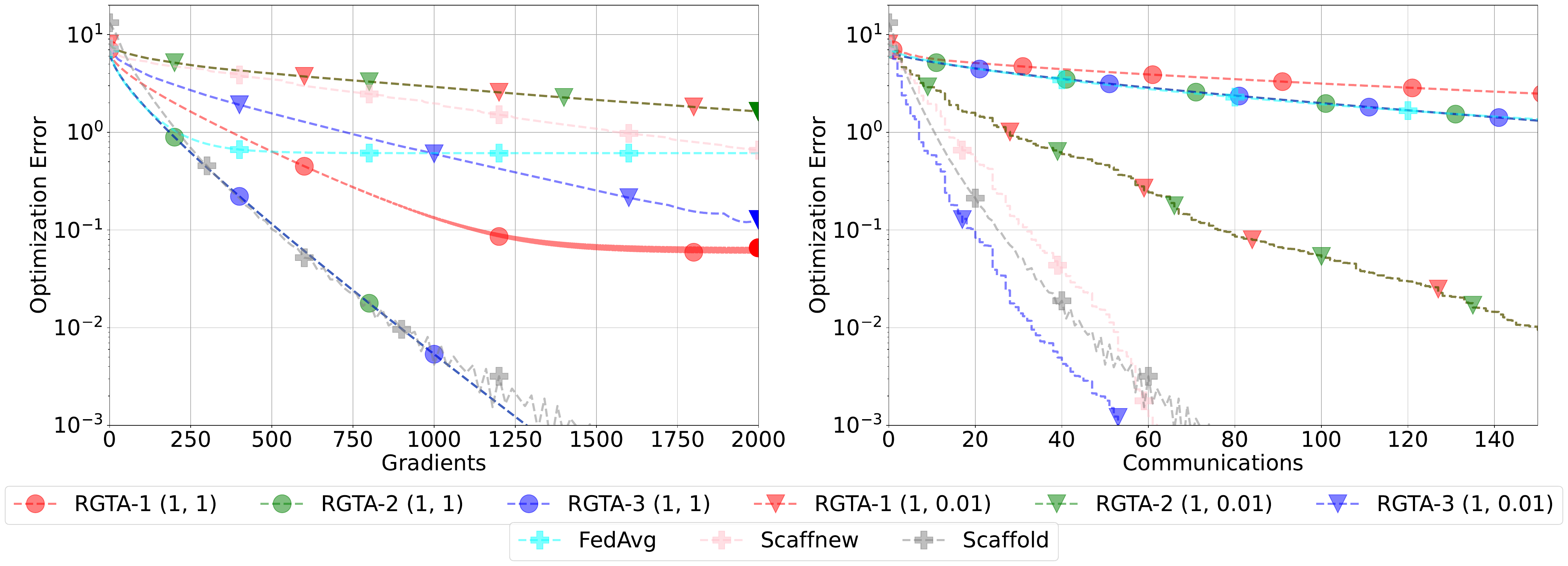} 
    \caption{Fully Connected Network, $\beta = 0$}
    \label{fig:w8a_full}
\end{subfigure}
\begin{subfigure}{\textwidth}
    \includegraphics[width=\textwidth]{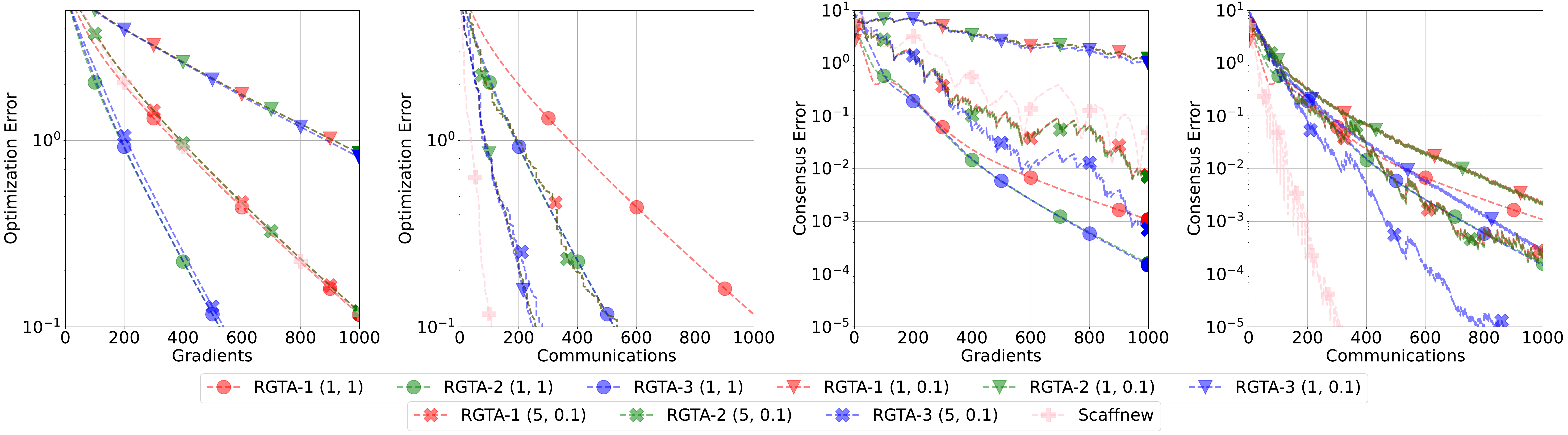} 
    \caption{Start Network, $\beta = 0.95$}
    \label{fig:w8a_star}
\end{subfigure}
\begin{subfigure}{\textwidth}
    \includegraphics[width=\textwidth]{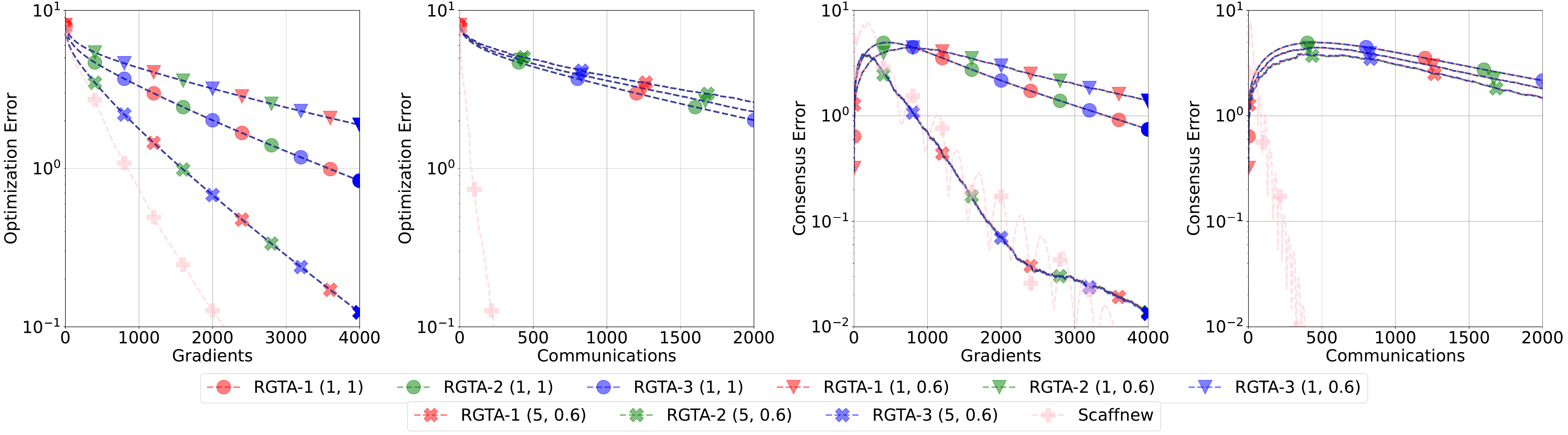} 
    \caption{Line Network, $\beta = 0.999$}
    \label{fig:w8a_line}
\end{subfigure}
\caption{Optimization Error ($\|\xbar_k - x^*\|_2$) and Consensus Error ($\|\xmbf_k - \xbb_k\|_2$) of \texttt{RGTA-1}, \texttt{RGTA-2} and \texttt{RGTA-3} with respect to number of gradient evaluations and communication rounds for binary logistic regression on the w8a dataset ($n = 16$, $d = 300$, $\sum_{i=1}^n n_i = 49,749 $) over; (a) fully connected network ($\beta =  0$); (b) star network ($\beta = 0.95$); and, (c) line network ($\beta = 0.999$).}
\label{fig:w8a_exps}
\end{figure}

%%%%%%%%%%%%%%%%%%%%%%%%%%%%
% Conclusion
%%%%%%%%%%%%%%%%%%%%%%%%%%%%
\section{Conclusions}\label{sec.conc}

In this paper, we have proposed a flexible gradient tracking algorithmic framework (\texttt{RGTA}). 
The framework allows one to balance the number of communication and computation steps in expectation for solving decentralized optimization problems. As the complexity of these two steps can differ significantly across applications, employing such flexibility can result in improvements for overall practical performance. The framework is endowed with this flexibility via a randomized scheme with two parameters (the number of communication steps and the communication probability). Moreover, the framework recovers popular gradient tracking methods as special cases and allows for a direct theoretical comparison. We have established a linear rate of convergence in expectation for the \texttt{RGTA} framework and showed reduction in communication and computation complexity as a direct result of the flexibility. Finally, we illustrated the benefits of utilizing the flexibility of the \texttt{RGTA} framework empirically on synthetic quadratic problems and logistic regression problems, over several network structures. We compared against popular algorithms from the literature and illustrated the robustness and efficiency of the proposed framework.

\newpage
\bibliographystyle{plain}
\bibliography{references}

%%%%%%%%%%%%%%%%%%%%%%%%%%%%
% Appendix
%%%%%%%%%%%%%%%%%%%%%%%%%%%%
% \newpage
% \appendix
% \input{appendix.tex}

\end{document}